\documentclass{article}

\usepackage[utf8]{inputenc} 
\usepackage[T1]{fontenc}    
\usepackage{url}            
\usepackage{booktabs}       
\usepackage{amsfonts}       
\usepackage{nicefrac}       
\usepackage{microtype}      
\usepackage{lipsum}		
\usepackage{graphicx}
\usepackage{natbib}
\usepackage{doi}

\usepackage{amsthm,amsmath,amssymb}
\usepackage{enumerate}
\usepackage{etoolbox}
\usepackage{mathtools}
\usepackage[table]{xcolor}
\usepackage{tcolorbox}
\usepackage{tikz-cd}
\usepackage{tikz}
\usetikzlibrary{shapes.geometric, arrows.meta, positioning}
\usepackage{changepage}
\usepackage{algpseudocodex}
\usepackage{multirow}
\usepackage{tabularx}
\usepackage{makecell}

\usepackage[nameinlink]{cleveref}
\creflabelformat{equation}{#2\textup{#1}#3}

\newcommand{\Z}{\mathbb{Z}}

\newcommand{\C}{\mathbb{C}}

\newcommand{\Q}{\mathbb{Q}}
\newcommand{\Qbar}{\overline{\Q}}
\newcommand{\Qtr}{\Q^{tr}}

\newcommand{\R}{\mathbb{R}}
\newcommand{\Ralg}{\R_{alg}}

\newcommand{\Sf}{\mathfrak{S}}
\newcommand{\Tf}{\mathfrak{T}}
\newcommand{\pf}{\mathfrak{p}}
\newcommand{\Pf}{\mathfrak{P}}

\newcommand{\K}{\mathbb{K}}
\newcommand{\F}{\mathbb{F}}

\newcommand{\OO}{\mathcal{O}}

\newcommand{\innerproduct}[2]{\langle #1, #2 \rangle}
\renewcommand{\bold}{\boldsymbol}

\DeclareMathOperator{\Gal}{Gal}

\newtheorem{theoremcount}{Theoremcount}[section]
\newtheorem{conjcount}{Conjcount}

\theoremstyle{plain}
\newtheorem{theorem}[theoremcount]{Theorem}
\newtheorem*{theorem*}{Theorem}
\newtheorem{lemma}[theoremcount]{Lemma}
\newtheorem{proposition}[theoremcount]{Proposition}
\newtheorem{corollary}[theoremcount]{Corollary}
\newtheorem{conjecture}[conjcount]{Conjecture}

\theoremstyle{definition}

\newtheorem{example}[theoremcount]{Example}

\crefname{theoremcount}{Theorem}{Theorems}
\crefname{conjcount}{Conjecture}{Conjectures}
\crefname{conj1count}{Conjecture}{Conjectures}
\crefname{quescount}{Question}{Questions}
\crefname{theorem}{Theorem}{Theorems}
\crefname{lemma}{Lemma}{Lemmas}
\crefname{proposition}{Proposition}{Propositions}
\crefname{corollary}{Corollary}{Corollaries}
\crefname{definition}{Definition}{Definitions}
\crefname{example}{Example}{Examples}
\crefname{remark}{Remark}{Remarks}
\crefname{conjecture}{Conjecture}{Conjectures}
\crefname{question}{Question}{Questions}

\date{\today}

\author{ \href{https://orcid.org/0009-0001-0062-648X}{\includegraphics[scale=0.06]{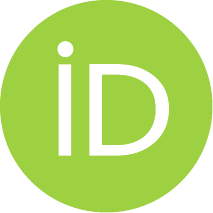}}\hspace{1mm}Igor V.~Loo \thanks{School of Physical and Mathematical Sciences, Nanyang Technological University, Singapore, email:
\texttt{IGOR002@e.ntu.edu.sg}.
The work of Igor Van Loo was supported by an NTU Provost Graduate Award, whose support is gratefully acknowledged.} \hspace{2mm} and 
\href{https://orcid.org/0000-0003-3141-3118}{\includegraphics[scale=0.06]{orcid.pdf}}\hspace{1mm}Fr\'ed\'erique Oggier
\thanks{School of Mathematics, University of Birmingham, UK.}
}

\hypersetup{
pdftitle={Real Fiducial Vectors Only Exist in Dimensions d = 3 (mod 4)},
pdfsubject={q-bio.NC, q-bio.QM},
pdfauthor={Igor V.~Loo, Frederique E.~Oggier},
pdfkeywords={SIC-POVMs; Equiangular Lines; Real Fiducial Vectors},
}

\begin{document}
\title{Real Fiducial Vectors Only Exist in Dimensions $d\equiv 3 \pmod{4}$}

\maketitle

\begin{abstract}
We investigate restrictions on the dimensions and coordinate fields of real Weyl-Heisenberg fiducial vectors. Our main result is that such vectors can exist only in dimensions $d\equiv 3\pmod{4}$. This result is motivated by evidence showing that real Weyl-Heisenberg fiducial vectors can exist only in dimensions $d = 4m^2 + 3, m \in \Z_{\geq 0}$. We further restrict the possible dimensions when the coordinate field is a totally real number field, an even or odd degree extension of $\Q$, a real quadratic field, or a maximal real subfield of a cyclotomic field. Lastly, we establish non-existence over odd-degree number fields under either positivity of all non-trivial cyclic-shift overlaps or a conjectural containment of a distinguished quadratic field in the SIC field.
\end{abstract}

{\small

}

%
%
\section{Introduction}

A maximal set of $d^2$ equiangular lines in $\C^d$ is the underlying mathematical object defining what is known in quantum physics as a symmetric informationally complete positive-operator-valued measure (SIC-POVM).

The existence of a SIC-POVM in every dimension is known as Zauner's conjecture \cite{zauner_2011} where he further postulates that a SIC-POVM is the orbit, under the Weyl-Heisenberg group, of a single unit vector, and since then numerous works have provided numerical and exact constructions, further strengthening the conjecture (see, e.g., \cite{Renes_2004, scott_grassl_2009, scott2017sicsextendinglistsolutions, Grassl_2017_FibLucas, bengtsson2024sicpovmsstarkunitsdimensions, Appleby_2022_StarkUnit}). 

Furthermore, in \cite{appleby2025constructiveapproachzaunersconjecture}, a proof of Zauner's conjecture was given, conditional on two conjectures: the order-one Abelian Stark conjecture, and the Twisted Convolution conjecture. The latter was recently proven \cite{radchenko2026realquadraticfieldsfinite, appleby2026twistedconvolutionidentityghost}, and \cite{starkreciprocity2026} establishes the relevant Stark reciprocity laws for real quadratic fields, thereby proving Zauner's conjecture. We refer the reader to \cite{Fuchs_2017} for an excellent overview on the topic as a whole.

We now fix some notation. Let $\Z_d$ be the integers modulo $d$ and $I_d$ be the $d$-dimensional identity matrix. 
Set $\omega_d = e^{2\pi i/d}$ and define the $d$-dimensional unitary matrices

\begin{equation}
    \bold{Z} = 
    \begin{pmatrix}
        1 & 0 & 0 & \cdots & 0 \\
        0 & \omega_d & 0 & \cdots & 0 \\
        0 & 0 & \omega_d^2 & \cdots & 0 \\
        \vdots & \vdots & \vdots & \ddots & \vdots \\
        0 & 0 & 0 & \cdots & \omega_d^{d - 1} \\
    \end{pmatrix}, \quad
    \bold{X} = 
    \begin{pmatrix}
        0 & 0 & 0 & \cdots & 1 \\
        1 & 0 & 0 & \cdots & 0 \\
        0 & 1 & 0 & \cdots & 0 \\
        \vdots & \vdots & \ddots & \ddots & \vdots \\
        0 & 0 & \cdots & 1 & 0
    \end{pmatrix}.
\end{equation}

A \textit{Weyl-Heisenberg fiducial vector}, or more simply, a \textit{fiducial vector} $\bold{v} = (v_0, \ldots, v_{d - 1})^T \in \C^d$ is a unit vector such that the $d^2$ vectors $\bold{v}_{a, b} = \bold{X}^a\bold{Z}^b\bold{v}$, $a,b\in\Z_d$, represent a set of $d^2$ equiangular lines, that is, they satisfy the following equation:

\begin{equation}
\label{eq:wh_fid_cond}
    |\innerproduct{\bold{v}_{a, b}}{\bold{v}}|^2 = \begin{cases}
        1, \quad a = b = 0, \\
        \frac{1}{d + 1}, \quad (a, b) \in \Z_d^2 \backslash\{(0, 0)\}.
    \end{cases}
\end{equation}
A particular case of a fiducial vector is a \textit{real fiducial vector} when all the vector entries are real, that is $\bold{v}\in\R^d$. Such vectors are known to exist (see, e.g., \cite{khatirinejad2008} for examples in dimensions $7$ and $19$). In this case, \Cref{eq:wh_fid_cond} simplifies to
\begin{equation}
\label{eq:wh_fid_cond_real}
    \innerproduct{\bold{v}_{a, b}}{\bold{v}}^2 = \begin{cases}
        1, \quad a = b = 0, \\
        \frac{1}{d + 1}, \quad b=0,~a \in \Z_d \backslash\{0\} \\
        \frac{e^{2i\theta_{a,b}}}{d + 1}, \quad else.
    \end{cases}
\end{equation}

A fiducial vector has an associated \textit{fiducial projector} $\Pi(\bold{v}) = \bold{v}\bold{v}^{\dag}$.

Beyond proving the existence of a fiducial vector in every dimension, the many connections with algebraic number theory \cite{Appleby_2017, appleby2012galoisautomorphismssymmetricmeasurement, Bengtsson_2017} introduce another interesting area of research whose goal is to find where the coordinates of these fiducial projectors exist. 
To do so, what is known as the SIC field \cite[Definition 1.10]{appleby2025constructiveapproachzaunersconjecture}, defined by
\begin{equation}
   E_{\Pi(\bold{v})} = \Q(\omega_{2d}, \Pi(\bold{v})_{jl} : 0\leq j, l < d)
\end{equation}
was introduced. The fact that the SIC field is predicted to be a number field is no surprise since numerous exact examples of fiducial vectors have been found with coefficients in some number field $K$ \cite{Appleby_2018_ConstructExact}. Furthermore, given the existence of a fiducial vector in $\C^d$, one can always find a fiducial vector in $K^d$ as shown in \cite{VanLooOgg2026}, where $K$ is called the \textit{coordinate field} of the fiducial vector. 

The study of SIC fields has led to several conjectures, such as the so-called \textit{Ray Class Field Conjecture}:
\begin{conjecture}
\label{conj:ray_class_field_conj}
    Let $d > 3$, and let
    \begin{equation}
        \K_d = \Q(\sqrt{(d + 1)(d - 3)}).
    \end{equation}
    Then $E_{\Pi(\bold{v})}$ is an Abelian extension of $\K_d$. Furthermore, there is at least one SIC-POVM whose corresponding SIC field is equal to $\mathfrak{R}$, the ray class field of $\K_d$. Moreover, every other SIC field in dimension $d$ is a finite Galois extension of $\Q$ containing $\mathfrak{R}$.
\end{conjecture}


The above conjecture applies to $E_{\Pi(\bold{v})}$ irrespective of whether $\bold{v}\in \C^d\backslash \R^d$ or $\bold{v}\in\R^d$. The structure of real fiducial vectors is, however, expected to be more constrained than that of generic fiducial vectors. Indeed, Scott observed a symmetry pattern suggesting that real fiducial vectors exist in every dimension $d=4m^2+3, m \in \Z_{\geq 0}$
\cite[Page 3]{scott2017sicsextendinglistsolutions},
which is supported by subsequent Stark-unit constructions \cite{bengtsson2025conjectureflatsicpovms, Appleby_2022_StarkUnit}: in \cite{Appleby_2022_StarkUnit}, constructions of both standard and real fiducial vectors are found in dimension $d=n^2+3$, where $n$ is allowed to be odd or even, while in \cite{bengtsson2025conjectureflatsicpovms}, the construction of so-called almost flat fiducial vectors is obtained from real fiducial vectors through a Fourier transform.  
These observations suggest both that real fiducial vectors exist when $d=4m^2+3, m \in \Z_{\geq 0}$, and that they do not exist otherwise. We are interested in the latter.

\begin{conjecture}
\label{conj:real_fid_vec_in_4m2+3}
    Real fiducial vectors only exist in dimensions $d = 4m^2 + 3, m \in \Z_{\geq 0}$.
\end{conjecture}

We start by looking generally at real fiducial vectors in $\R^d$ in \Cref{sec:real}, where we first provide a new proof that $d$ cannot be even. We further rule out the special family of so-called Zauner fiducial vectors \cite{zauner_2011} by proving that they only exist in $\R^3$. We conclude this section by arguing that if a real fiducial vector exists in $\R^d$, it must exist in $K^d$ for some number field $K\subset \R$.

The claim made in the title is proven in \Cref{sec:dequiv3}. The main ingredients are identities derived from \Cref{eq:wh_fid_cond}, and an equivalent formulation of \Cref{eq:wh_fid_cond} with rational coefficients from \cite{khatirinejad2008}, together with computations of the valuation of a prime $\Pf$ in $K$ above 2.

The number fields $K\subset\R$ include the totally real number fields. We show in \Cref{sec:totallyreal} that if $K$ is totally real, then $d=4m^2-1$, $m\in\Z_{\geq 1}$. 

To gain further knowledge about $d$, we then restrict to specific number fields in \Cref{sec:fieldspecific}: first we treat those with tame ramification above $2$, before distinguishing between number fields of odd and even degree. Particular cases include quadratic extensions and maximal real subfields of cyclotomic fields.  

Lastly, in \Cref{sec:conditional_results}, we discuss how some of the questions discussed are completely resolved with additional assumptions which are conjectured to be true.

\Cref{fig:pd3mod4} in the appendix summarizes the main results of the paper.

%
%
\section{Real Fiducial Vectors}
\label{sec:real}

Real fiducial vectors, that is, fiducial vectors whose coefficients exist in $\mathbb{R}$, are known to exist, and furthermore, they are expected to be found in all dimensions $d = 4m^2 + 3, m \in \Z_{\geq 0}$ \cite[Page 3]{scott2017sicsextendinglistsolutions}. The smallest dimension in which such a vector is known is $d=3$.

\begin{example}
\label{ex:d=3}
The vector    
\begin{equation*}
    \bold{v} = \frac{1}{\sqrt{2}}\begin{pmatrix}
        0 & 1 & - 1
    \end{pmatrix}^T \in \R^3
\end{equation*}
is a fiducial vector \cite[Example $d = 3$]{zauner_2011}.
\end{example}

A particularly simple example was found in dimension seven.

\begin{example}
\label{example:d7_vector}
    In dimension $d = 7$, 
    set
    \begin{align*}
        a &= -\frac{1}{2} \sqrt{\frac{3\sqrt{2} - 2}{7}} \approx -0.283 \\
        b &= \frac{1}{2}\sqrt{\frac{10 - \sqrt{2} + \sqrt{14 + 28\sqrt{2}}}{14}} \approx 0.533 \\
        c &= \frac{1}{2}\sqrt{\frac{10 - \sqrt{2} - \sqrt{14 + 28\sqrt{2}}}{14}} \approx 0.150.
    \end{align*}
    Then $\bold{v} = (a, b, b, c, b, c, c)^T \in \R^7$ is a fiducial vector \cite[Theorem 5.2]{khatirinejad2008}.
\end{example}

Another example in dimension $d=19$ is known \cite[Theorem 5.3]{khatirinejad2008}. We note that
\[
3 = 4 \cdot 0^2 +3,~ 7 = 4 \cdot 1^2 +3,~ 19 = 4 \cdot 2^2 +3.
\]

Many more exact and numerical real fiducial vectors have been found in this family of dimensions $4m^2+3$, $m\in\Z_{\geq 0}$ \cite[Table III]{Appleby_2022_StarkUnit}, with the largest known real fiducial vector being found in dimension $19603 = 4 \cdot 70^2 +3$.

It is known that real fiducial vectors only exist in odd dimensions \cite[Proposition 7]{Caro_P_rez_2022}. We provide an alternative proof of this result. 

\begin{theorem}
\label{theorem:no_real_fiducial_in_even_dimensions}
    Real fiducial vectors cannot exist in even dimensions.
\end{theorem}
\begin{proof}
    Recall that $(\bold{v}_{a, b})_k = (\bold{X}^a \bold{Z}^b \bold{v})_k = \omega_d^{b(k - a)}v_{k - a \pmod{d}}$, $0\leq k \leq d-1$. We will drop the explicit mention of the modulo for the sake of simplicity in what follows. 
    
    Now suppose that we have $\bold{v} \in \R^d$ where $d$ is even, so that $a=d/2$ is an integer. We have 
    \begin{equation*}
        (\bold{v}_{d/2, 1})_k = \omega_d^{k - d/2}v_{k - d/2}
    \end{equation*} and therefore
    \begin{equation*}
        \innerproduct{\bold{v}_{d/2, 1}}{\bold{v}} = \sum_{j = 0}^{d - 1} \omega_d^{j - d/2} v_{j - d/2} \overline{v_j}.
    \end{equation*}
    Since $\bold{v}$ is a real fiducial vector, we have $\overline{v_j} = v_j$ and since $\omega_d$ is a primitive $d$th root of unity, $\omega_d^{d/2} = -1$, hence we have
    \begin{equation*}
        \innerproduct{\bold{v}_{d/2, 1}}{\bold{v}} = -\sum_{j = 0}^{d - 1} \omega_d^{j} v_{j - d/2} v_j.
    \end{equation*}
    Now notice that the $j$th term and the $(j - d/2)$th term of the sum cancel out. Explicitly, we have
    \begin{align*}
        \text{$j$th term} &= \omega_d^{j} v_{j - d/2} v_j, \\
        \text{$(j - d/2)$th term} &= \omega_d^{j - d/2} v_{j - d/2 - d/2} v_{j - d/2} = - \omega_d^{j} v_jv_{j - d/2}.
    \end{align*}
    
    Since the $d$ indices can be grouped into exactly $\frac{d}{2}$ pairs $(j, j - d/2)$, this implies that $\innerproduct{\bold{v}_{d/2, 1}}{\bold{v}} = 0$, which implies $|\innerproduct{\bold{v}_{d/2, 1}}{\bold{v}}|^2 = 0$ which contradicts the requirement in (\ref{eq:wh_fid_cond}) that $|\innerproduct{\bold{v}_{d/2, 1}}{\bold{v}}|^2 = \frac{1}{d + 1}$. Hence a real fiducial vector cannot exist in any even dimension.
\end{proof}

A particular type of fiducial vector is a Zauner fiducial vector \cite{zauner_2011}, that is a vector $\bold{v}\in\C^d$ that is both a fiducial vector, and it is an eigenvector of the Zauner matrix $\bold{Z}_d$ defined as $\bold{Z}_d = \alpha_d \bold{F}_d \bold{G}_d$, where
\[
\alpha_d = e^{\frac{i\pi (d - 1)}{12}},~
    (\bold{F}_d)_{jk} = \frac{1}{\sqrt{d}} \omega_d^{jk},~
    (\bold{G}_d)_{jj} = g_j = \omega_{2d}^{j^2(d + 1)}.
\]

The eigenvalues of $\bold{Z}_d$ are the cube roots of unity: $1, \omega_3, \omega_3^2$ \cite[Equation 3.11]{zauner_2011}. We prove next that a real Zauner fiducial vector only exists for $d=3$, and such a vector is known to exist. Indeed, the vector of \Cref{ex:d=3} is a real Zauner fiducial vector. 

\begin{proposition}
\label{prop:no_real_zauner_fiducial}
    Let $d > 1$ and $\bold{v} \in \R^d$ be a fiducial vector such that $\bold{Z}_d\bold{v} = \lambda\bold{v}$ where $\lambda^3 = 1$. Then $d = 3$.
\end{proposition}
\begin{proof}
    Let $\bold{v} = (v_0, \ldots, v_{d - 1})^T$, and define
    \begin{equation*}
        P(X) = \sum_{j = 0}^{d - 1} v_jX^j \in \R[X].
    \end{equation*}
    Notice that
    \begin{equation*}
        \bold{F}_d^{\dag} \bold{v} = \bold{F}_d^{\dag} \bold{Z}_d \frac{\bold{v}}{\lambda} = \alpha_d\lambda^2 \bold{G}_d \bold{v},
    \end{equation*}
    using for the second equality that $\bold{F}_d^{-1}=\bold{F}_d^\dagger$ and that $\lambda^{-1}=\lambda^2$, which, written componentwise, becomes
    \begin{equation*}
        \frac{1}{\sqrt{d}} \sum_{j = 0}^{d - 1}v_j \omega_d^{-jk} = (\bold{F}_d^{\dag} \bold{v})_k = \alpha_d \lambda^2  g_k v_k,~0 \leq k \leq d-1.
    \end{equation*} 
    Using $P(X)$, we have
    \begin{align}
    \begin{split}
    \label{eq:Pomegad-k}
        P(\omega_d^{-k}) = \sum_{j = 0}^{d - 1} v_j\omega_d^{-jk} 
        = \sqrt{d}(\bold{F}_d^{\dag} \bold{v})_k 
        = \sqrt{d} (\alpha_d\lambda^2  g_k v_k).
    \end{split}
    \end{align}
    Now, we compute $\overline{P(\omega_d^{-k})}$ in two ways. Firstly, we have $\overline{P(\omega_d^{-k})} = \sqrt{d}  (\overline{\alpha}_d \lambda g_k^{-1} v_k)$.  Secondly, we have 
    \begin{equation*}
        \overline{P(\omega_d^{-k})} = \overline{\sum_{j = 0}^{d - 1} v_j\omega_d^{-jk}} 
        = \sum_{j = 0}^{d - 1} \overline{v_j} \overline{\omega_d^{-jk}} 
        = \sum_{j = 0}^{d - 1} v_j\omega_d^{jk} 
        = P(\omega_d^k).
    \end{equation*}
    Therefore, substituting $-k$ into \Cref{eq:Pomegad-k}, and noting that $g_{-k} = g_k$, we get
    \begin{align*}
        \sqrt{d}(\alpha_d\lambda^2 g_{-k} v_{-k}) &= P(\omega_d^k) = \overline{P(\omega_d^{-k})} = \sqrt{d}(\overline{\alpha_d}\lambda g_{k}^{-1}v_{k}) \\
        \Rightarrow v_{-k} &= \lambda^2 \alpha_d^{-2} g_k^{-2} v_k.
    \end{align*}
    By \Cref{theorem:no_real_fiducial_in_even_dimensions}, we know $d$ is odd, and hence $g_k^2 = (\omega_{2d}^{k^2(d + 1)})^2 = \omega_d^{k^2}$ and so  $v_{-k} = \beta_k v_k$ for $\beta_k = \lambda^2 \alpha_d^{-2} \omega_d^{-k^2}$.

    From \Cref{eq:wh_fid_cond_real},
    \begin{equation*}
        \innerproduct{\bold{v}_{1, 0}}{\bold{v}}^2 = (\sum_{j = 0}^{d - 1} v_{j - 1}v_{j})^2 =  \frac{1}{d + 1} \neq 0,
    \end{equation*}
    so we must have some $v_kv_{k + 1} \neq 0$, which implies that $v_k$ and $v_{k + 1}$ are both non-zero.
    
    Since $v_k \neq 0$, $\frac{v_{-k}}{v_k} = \beta_k \in \R$, and since $\beta_k$ is a product of roots of unity, it must be equal to $\pm 1$, which implies $\beta_k^2 = 1$ and hence
    \begin{align*}
        \lambda^4 \alpha_d^{-4} \omega_d^{-2k^2} = 1 \Rightarrow 
        \lambda^{24} \alpha_d^{-24} \omega_d^{-12k^2} = 1 \Rightarrow
        \omega_d^{-12k^2} = 1 \Rightarrow 
        d &\mid 12k^2
    \end{align*}
    since $\lambda^{24} = \alpha_d^{24} = 1$. 

    Applying the same argument to $v_{k+1}\neq 0$, we get $d \mid 12k^2$ and $d \mid 12(k + 1)^2$, hence $d \mid \gcd(12k^2, 12(k + 1)^2)$ and we know
    \begin{equation*}
        \gcd(12k^2, 12(k + 1)^2) = 12\gcd(k^2, (k + 1)^2) = 12\gcd(k, k + 1)^2.
    \end{equation*}
    Since the $\gcd$ of two consecutive integers is one, we have $\gcd(12k^2, 12(k + 1)^2) = 12$, hence $d \mid 12$, and since $d > 1$ is odd, we must have $d = 3$.
\end{proof}

We conclude this section by observing that to study real fiducial vectors, it is enough to study fiducial vectors in real number fields. 
Let $\overline{\Q}$ be the algebraic closure of $\mathbb{Q}$, which contains all algebraic numbers. Then $\Ralg = \overline{\Q} \cap \mathbb{R}$ is the field of algebraic reals.

\begin{proposition}
    If there exists a fiducial vector in $\R^d$, there exists a fiducial vector in $\Ralg^d$.   
\end{proposition}
\begin{proof}
    Take the system of polynomial equations described in \cite[Section 2.2]{VanLooOgg2026}, which characterizes a fiducial vector in $\C^d$, and add the equations $B_j = 0$ for all $0 \leq j \leq d - 1$, which forces the imaginary part of each coefficient of the fiducial vector to be $0$. The assumed existence of a real fiducial vector would imply the existence of a real solution to this polynomial system of equations. Hence, by \cite[Theorem 3.2]{VanLooOgg2026}, there exists a solution in $\Ralg^d$, which in turn corresponds to the existence of a real algebraic fiducial vector.
\end{proof}

\begin{example}
\label{example:d7_fiducial_and_sic_field}
 We describe the number field in which all the coefficients of the fiducial vector $\bold{v}$ of \Cref{example:d7_vector} exist.
    Observe that
    \begin{align*}
        \sqrt{2} = \frac{12 - (56b^2 - 10)^2}{112b^2 - 48},~
        c = \frac{3\sqrt{2} - 2}{28b},~
        a = -\frac{b + c}{1 + \sqrt{2}},
    \end{align*}
    and hence $\bold{v} \in \Q(b)^7$ where $b$ has minimal polynomial $3136X^8 - 2240X^6 + 568X^4 - 56X^2  + 1$. We thus know that $[\Q(b) : \Q] = 8$ and it can be computed that $\Q(b)$ has a discriminant of $d_{\Q(b)}=2^{22} \cdot 7^2$. The subfield structure of $\Q(b)$ is illustrated in Figure \ref{fig:Q(b)}. It can further be computed that $\Q(b)$ has four real embeddings and two pairs of complex embeddings. Thus it is not totally real, which can be explicitly checked as follows. Since $a^2 = \frac{3\sqrt{2} - 2}{28}$, the embedding
    \begin{equation}
        \sigma : \Q(b) \rightarrow \C, \sigma(\sqrt{2}) = -\sqrt{2},
    \end{equation}
    yields $\sigma(a^2) = \frac{-3\sqrt{2} - 2}{28}= \sigma(a)^2  < 0$ and hence $\sigma(a) \notin \R$. We have therefore shown that $\Q(b)\subset \R_{alg}$ but $\Q(b)\not\subseteq \Q^{tr}$.

\begin{figure}[h]
    \centering
    \begin{equation*}
\begin{tikzcd}
& 
\begin{array}{c}
\Q(b) \\
d_{\Q(b)}=2^{22} \cdot 7^2 \\
\end{array}
\arrow[d, no head] \arrow[dl,no head]\arrow[dr,no head] & \\
\begin{array}{c}
K_1\\
d_{K_1}=-2^6 \cdot 7\\
\end{array}
\arrow[dr,no head] &
\begin{array}{c}
K_2\\
d_{K_2}= 2^{11}\\
\end{array}
\arrow[d,no head] &
\begin{array}{c}
K_3 \\
d_{K_3}=-2^{11}\cdot 7\\
\end{array}
\arrow[dl,no head] \\
&\Q(\sqrt{2})\arrow[d,no head] & \\
&\Q & \\
\end{tikzcd}
\end{equation*}
    \caption{The number field $\Q(b)$ in which the coefficients of the fiducial vector of \Cref{example:d7_vector} exist, with its subfields and their respective discriminants.}
    \label{fig:Q(b)}
\end{figure}
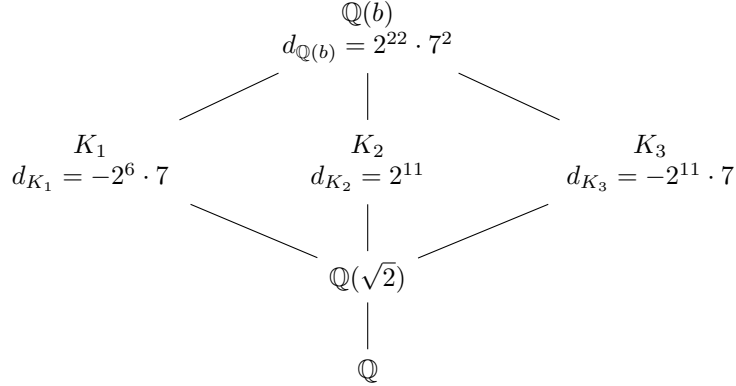

The number field $\Q(b)$ containing the coefficients of the fiducial vector $\bold{v}$ is not to be confounded with its  SIC field, usually studied in the literature \cite[Definition 1.10]{appleby2025constructiveapproachzaunersconjecture}. To further elaborate, we compute the SIC field of $\bold{v}$, which by definition is the number field $E_{\Pi(\bold{v})}=\Q(\omega_7, \Pi(\bold{v})_{jl}:0\leq j,l<7)$ where $\Pi(\bold{v})_{jl}$ are the entries of the projection matrix $\Pi(\bold{v}) = \bold{v}\bold{v}^\dagger$. Explicitly,
\begin{equation*}
    \Pi(\bold{v})=\boldsymbol{v}\boldsymbol{v}^{\dagger}
    =
    \begin{pmatrix}
        a^2 & ab  & ab  & ac  & ab  & ac  & ac \\
        ab  & b^2 & b^2 & bc  & b^2 & bc  & bc \\
        ab  & b^2 & b^2 & bc  & b^2 & bc  & bc \\
        ac  & bc  & bc  & c^2 & bc  & c^2 & c^2 \\
        ab  & b^2 & b^2 & bc  & b^2 & bc  & bc \\
        ac  & bc  & bc  & c^2 & bc  & c^2 & c^2 \\
        ac  & bc  & bc  & c^2 & bc  & c^2 & c^2
    \end{pmatrix}.
\end{equation*}
Hence, we have
\begin{equation*}
    E_{\Pi(\bold{v})}
    =\Q(\omega_7, \Pi(\bold{v})_{jl})
    =\Q(\omega_7, a^2,b^2,c^2,ab,ac,bc).
\end{equation*}

Set
 $   \theta=\sqrt{2\sqrt{2}-1}$.
We will show that $\Q(\omega_7, \theta) = E_{\Pi(\bold{v})}$. 
Notice that 
\begin{equation*}
    \theta = \frac{c - b}{a} = \frac{ac - ab}{a^2} \in \Q(\Pi(\bold{v})_{jl}),
\end{equation*} 
hence $\Q(\theta) \subseteq \Q(\Pi(\bold{v})_{jl})$.
Conversely, 
\begin{align*}
    \sqrt{2} &= \frac{\theta^2+1}{2}, 
    ~a^2 = \frac{3\sqrt{2}-2}{28}, \\
    \frac{b}{a} &= \frac{-(1+\sqrt{2})-\theta}{2}, ~\frac{c}{a} = \frac{-(1+\sqrt{2})+\theta}{2}.
\end{align*}
Thus, $a^2$, $b/a$, and $c/a$ all belong to
$\Q(\theta)$, and therefore every entry of $\Pi(\bold{v})$ belongs to $\Q(\theta)$, since 
\begin{align*}
    ab &= a^2  \frac{b}{a},~
    ac = a^2  \frac{c}{a},~
    b^2 = a^2 (\frac{b}{a})^2,~
    c^2 = a^2  (\frac{c}{a})^2,~
    bc = a^2  \frac{b}{a} \frac{c}{a}.
\end{align*}
It follows that
    $\Q(\Pi(\bold{v})_{jl}) = \Q(\theta)$ and $
    E_{\Pi(\bold{v})} = \Q(\omega_7,\theta)$.

This computation agrees with the SIC field for orbit $7b$ reported in 
\cite[Appendix A, Tables 2, 3, and 5]{appleby2012galoisautomorphismssymmetricmeasurement}, which we denote by $E_A$.
Indeed, substituting the reported field generators into the $7b$ row of
Table 5 gives
\begin{equation*}
    E_A = \Q(\sqrt{2},\cos\tfrac{\pi}{7},i\sqrt{2\sqrt{2}+1},i\sqrt{7}).
\end{equation*}
To compare $E_{\Pi(\bold{v})}$ and $E_A$, write
    $\eta = i\sqrt{2\sqrt{2}+1}$.
We know $\Q(\cos\tfrac{\pi}{7},i\sqrt{7}) = \Q(\omega_7)$.  
Then since
    $\theta\eta = i\sqrt{7}$,
we have
\begin{equation*}
    E_{A} =\Q(\omega_7,\eta) = \Q(\omega_7,\theta) = E_{\Pi(\bold{v})}
\end{equation*}
as claimed.

Since $E_{\Pi(\bold{v})}$ contains all four roots $\pm\theta,\pm\eta$ of
    $X^4+2X^2-7$
and contains all powers of $\omega_7$, $E_{\Pi\bold{v})}$ is the splitting field over $\Q$ of
\begin{equation*}
    (X^4+2X^2-7)(X^6+X^5+X^4+X^3+X^2+X+1),
\end{equation*}
and therefore is Galois over $\Q$.
Its discriminant is $d_{E_{\Pi\bold{v})}}=2^{36}7^{20}$. We will return to the connection between the number field containing the coefficients of the fiducial vector and its SIC field in \Cref{subsec:Kd}.
 
 
\end{example}

We note that computations in the above example, as well as in the whole paper, were done using SageMath \cite{sagemath}.

%
%
%

\section{The Case \texorpdfstring{$d \equiv 3 \pmod{4}$}{d = 3 (mod 4)}}
\label{sec:dequiv3}

The starting point for studying fiducial vectors with coefficients in a number field $K\subset \R$ is \Cref{eq:wh_fid_cond}, together with an equivalent formulation proposed in \cite[Theorem 2.2]{khatirinejad2008}. We start by deriving generalities from both these characterizations in the particular case of real fiducial vectors, which in turn will allow us to rule out the existence of real fiducial vectors in dimension $d\equiv 1 \pmod{4}$.

%
%

\subsection{Generalities}

Consider first the following equality from \Cref{eq:wh_fid_cond}:

\begin{equation*}
    c_a = \innerproduct{\bold{v}_{a, 0}}{\bold{v}} = \sum_{j = 0}^{d - 1} v_{j - a}\overline{v}_{j} = \pm \frac{e^{i\theta_{a, 0}}}{\sqrt{d + 1}},~a \neq 0.
\end{equation*}
Then $|c_a|^2 = \frac{1}{d + 1}$. If $\bold{v}$ is a real vector, the right-hand side of the above equation is also real (which results from the choice of $\bold{X}$ and not $\bold{Z}$) and hence $|c_a|^2 = c_a^2$, which implies $e^{i\theta_{a,0}} = \pm 1$, that is, $c_a = \pm \frac{1}{\sqrt{d + 1}}$.

Thus, we let $\sigma_a$ denote the sign of $c_a$ and call the following identity the $\bold{X}$-overlap equation,
\begin{equation}
\label{eq:Xoverlap}
    c_a = \innerproduct{\bold{v}_{a, 0}}{\bold{v}} = \sum_{j = 0}^{d - 1} v_{j - a}v_{j} = \frac{\sigma_a}{\sqrt{d + 1}}.
\end{equation}


We start by restricting the dimension of vectors whose inner product is $\pm\tfrac{1}{\sqrt{d+1}}$.

\begin{lemma}
\label{lemma:dimensions_4m2-1}
Let $K \subset \R$ be a number field, and $d > 1$.
\begin{enumerate}
\item 
Suppose $\bold{u}^T\bold{v}=\pm\tfrac{1}{\sqrt{d+1}}$ for $\bold{u},\bold{v}\in K^d$.
Then $\Q(\sqrt{d + 1}) \subseteq K$. Notably, if $\sqrt{d + 1} \in \Q$, then the dimension is restricted to $d + 1 = k^2$ where $k \in \Z_{\geq 1}$.
\item 
Fiducial vectors in $K^d$, where $\sqrt{d + 1} \in \Q$, can only exist in dimension $d = 4m^2 - 1$ where $m \in \Z_{\geq 1}$.
\end{enumerate}
\end{lemma}
    
\begin{proof}
\begin{enumerate}
    \item 
        We have $\bold{u}^T\bold{v} = \sum_{j = 0}^{d - 1} u_{j}v_j = \pm \frac{1}{\sqrt{d + 1}}$. Since the left-hand side is a sum of elements in $K$, it is an element of $K$, which forces the right hand side $\pm\frac{1}{\sqrt{d + 1}}$ to belong to $K$ as well, and it is clearly in $\Q(\sqrt{d + 1})$, hence $\Q(\sqrt{d + 1}) \subseteq K$. 
    \item 
        Suppose $\bold{v} \in K^d$. Since $\bold{X}^a$ simply permutes the coordinates of $\bold{v}$, for any $a \neq 0$, let $\bold{u} = \bold{v}_{a, 0} \in K^d$, then
        \begin{equation*}
            \bold{u}^T\bold{v} = c_a =\sum_{j = 0}^{d - 1} v_{j - a}v_{j} = \pm \frac{1}{\sqrt{d + 1}}
        \end{equation*}
        which fits the requirements of point (1). 
        Furthermore, if $k = 2m + 1$ is odd, then $d = 4m^2 + 4m$ is even, in which case no real fiducial vector can exist by \Cref{theorem:no_real_fiducial_in_even_dimensions}. Hence we are restricted to $k = 2m$ which implies $d = 4m^2 - 1$.
\end{enumerate}
\end{proof}

We remark that point (1) of \Cref{lemma:dimensions_4m2-1} constrains the dimensions of a maximal equiangular set of lines, even if it is not generated by a fiducial vector. We also remark that this lemma warrants investigation into when $\sqrt{d + 1}$ is or is not in $\Q$.

We define the following notation which will be used throughout the paper:
\begin{align}
    \Sf &= \sum_{a = 1}^{\tfrac{d - 1}{2}} \sigma_a \in \Z,\\
    \Tf &= \sum_{j=0}^{d-1}v_j \in K.
\end{align}
Note that these definitions are all tied to a fiducial vector $\bold{v} \in K^d \subset \R^d$, and since the existence of a real fiducial vector requires the dimension $d$ to be odd, $\tfrac{d-1}{2}$ is a positive integer. 

\begin{lemma}
\label{lemma:Tf_equation}
 For $\bold{v} \in K^d \subset \R^d$, we have \begin{equation*}
     \Tf^2 = 1 + \frac{2\Sf}{\sqrt{d+1}}.
 \end{equation*}
\end{lemma}
\begin{proof}
    We have $c_a = \frac{\sigma_a}{\sqrt{d + 1}}$. Note that $c_a = c_{-a}$ since $c_{-a} = \sum_{j = 0}^{d - 1} v_{j 
    + a}v_{j} = \sum_{k = 0}^{d - 1} v_kv_{k - a}$ where $k = j + a$, which works since $0 \leq k < d \pmod d$. This implies $\sigma_a = \sigma_{-a}$. 
    Now we sum $c_a$, $0\leq a \leq d-1$, in two different ways:
    \begin{equation*}
         \sum_{a = 0}^{d - 1} c_a 
            = \sum_{j = 0}^{d - 1}v_j \sum_{a = 0}^{d - 1} v_{j - a} 
            = \sum_{j = 0}^{d - 1} v_j \sum_{k = 0}^{d - 1} v_k 
            = \Tf^2, 
    \end{equation*}
    and since $\sigma_a = \sigma_{-a}$ and $\sum_{a=1}^{\tfrac{d-1}{2}}\sigma_a = \sum_{a=\tfrac{d-1}{2}+1}^{d-1}\sigma_a$, we further have     
    \begin{equation*}
      \sum_{a = 0}^{d - 1} c_a = c_0 + \sum_{a = 1}^{d - 1} c_a  = 1 + \frac{1}{\sqrt{d+1}} \sum_{a = 1}^{d - 1} \sigma_a 
      = 1 + \frac{2 \Sf}{\sqrt{d+1}}.
    \end{equation*}
    
    Combining both reformulations gives
    \begin{equation*}
        \Tf^2 = 1 + \frac{2 \Sf}{\sqrt{d+1}}.
    \end{equation*}    
\end{proof}

A priori, $\Sf$ ranges over $[-\tfrac{d-1}{2},\tfrac{d-1}{2}]$, but since $\Tf^2\geq 0$, this forces $\Sf\geq -\tfrac{\sqrt{d+1}}{2}$. 

\begin{lemma}
\label{lem:Sfodd}
\begin{enumerate}
    \item The sum 
 $\Sf = \sum_{a = 1}^{\tfrac{d - 1}{2}} \sigma_a $ and $\tfrac{d-1}{2}$ have the same parity.
    \item For $d \equiv 3 \pmod{4}$, $\Sf$ is odd. Notably, if $d = 4m^2 - 1, m \in \Z_{\geq 1}$, $\Sf$ is odd.
\end{enumerate}
\end{lemma}
\begin{proof}
\begin{enumerate}
    \item 
   Write $\tfrac{d-1}{2}=2l+\epsilon$, $l\in\Z_{\geq 1}$, where $\epsilon=1$ if  $\tfrac{d-1}{2}$ is an odd integer, and $\epsilon=0$ if  $\tfrac{d-1}{2}$ is even. We have
    \begin{equation*}
        \Sf = \sum_{a = 1}^{\tfrac{d - 1}{2}} \sigma_a = \Sf_+ - \Sf_-,
    \end{equation*}
    where $\Sf_+$, and $\Sf_-$ are the number of positive and negative $\sigma_a$, respectively. 
    Now, since $\Sf_++\Sf_-=2l+\epsilon$, we have
    \[
    \Sf = \Sf_+ - \Sf_- = (2l+\epsilon-\Sf_-) -\Sf_-=  2l+\epsilon-2\Sf_- 
    \]
    and $\Sf$ must have the same parity as $\tfrac{d-1}{2}$. 
     \item Let $d = 4n + 3, n \in \Z_{\geq 0}$. Then $d - 1 = 2(2n + 1)$ is even, hence $\tfrac{d - 1}{2} = 2n + 1$ is odd.
\end{enumerate}
\end{proof}

Consider next a reformulation of  \Cref{eq:wh_fid_cond} which states \cite[Theorem 2.2]{khatirinejad2008} that a vector $\bold{v} = \begin{pmatrix}
        v_0 & v_1 & \cdots & v_{d - 1}
    \end{pmatrix}^T \in \C^d$ is a fiducial vector if and only if for every $(s, t) \in \Z_d^2$ the following identity holds:
    \begin{equation}
        f_{s,t}(\bold{v}) = \sum_{j = 0}^{d - 1} v_j\overline{v}_{j+s}v_{j+t}\overline{v}_{j+s+t} = \frac{\delta_{s,0}+\delta_{t,0}}{d+1},
    \end{equation}
where $\delta_{u,0}=1$ if $u=0$ and $\delta_{u,0}=0$ otherwise for $u\in\{s,t\}$. 

For the real case, we can drop the conjugation; explicitly, we have
\begin{equation}
\label{eq:khat_real_fst}
    f_{s,t}(\bold{v}) = \sum_{j = 0}^{d - 1} v_jv_{j+s}v_{j+t}v_{j+s+t} = \frac{\delta_{s,0}+\delta_{t,0}}{d+1}.
\end{equation}

Let $\bold{v} \in K^d$ be a real fiducial vector. Recall that $d$ is odd by \Cref{theorem:no_real_fiducial_in_even_dimensions} so $(d-1)/2$ is an integer. 
For $0 \leq r \leq (d-1)/2$, define the vector $\bold{w}_r$ as 
\begin{equation*}
    (\bold{w}_r)_j = v_{j+r}v_{j-r},~0\leq j \leq d-1.
\end{equation*}

We prove two lemmas regarding the inner products of $\bold{w}_0,\ldots,\bold{w}_{(d-1)/2}$. 

\begin{lemma}
\label{lemma:w_r_are_mutually_orthogonal}
    The vectors $\bold{w}_0,\ldots,\bold{w}_{(d-1)/2}$ are mutually orthogonal.   
\end{lemma}
\begin{proof}
    For $0 \leq r,q \leq (d-1)/2$, we have
    \begin{align*}
        \innerproduct{\bold{w}_r}{\bold{w}_q} &= \sum_{j = 0}^{d - 1} v_{j+r}v_{j-r}v_{j+q}v_{j-q}\\
        &= \sum_{k = 0}^{d - 1} v_{k+2r}v_kv_{k+r+q}v_{k+r-q}, \quad k = j - r\\
        &= \sum_{k = 0}^{d - 1} v_kv_{k+(r+q)}v_{k+(r-q)}v_{k+(r+q)+(r-q)}\\
        &= f_{r+q,r-q}(\bold{v}).
    \end{align*}
    Since
    \begin{equation*}
        -\frac{d-1}{2} \leq r-q \leq \frac{d-1}{2},
        \qquad
        0 \leq r+q \leq d-1,
    \end{equation*}
    we have $r-q \equiv 0 \pmod d$ if and only if $r=q$, and $r+q \equiv 0 \pmod d$ if and only if $r=q=0$. Therefore, following \Cref{eq:khat_real_fst}, we have
    \begin{equation*}
        \innerproduct{\bold{w}_r}{\bold{w}_q} =
        \begin{cases}
            \frac{2}{d+1}, & r=q=0,\\
            \frac{1}{d+1}, & r=q>0,\\
            0, & r\neq q.
        \end{cases}
    \end{equation*}
    In particular, the vectors $\bold{w}_0,\ldots,\bold{w}_{(d-1)/2}$ are mutually orthogonal.
\end{proof}

\begin{lemma}
\label{lemma:innerprod_of_w_r_and_all_ones}
    \begin{equation*}
        \innerproduct{\bold{1}}{\bold{w}_r} = \begin{cases}
            1 & r = 0 \\
            \tfrac{\sigma_{2r}}{\sqrt{d + 1}} & r \neq 0
        \end{cases},
    \end{equation*}
    where $\bold{1}=(1,\ldots,1)^T$ is the all-ones vector.
\end{lemma}
\begin{proof}
    The $j$th coordinate of $\bold{w}_0$ is $v_j^2$, so the normalization condition from \Cref{eq:wh_fid_cond} gives
    \begin{equation*}
        \innerproduct{\bold{1}}{\bold{w}_0}
        = \sum_{j = 0}^{d - 1}v_j^2
        = 1.
    \end{equation*}
    For $1 \leq r \leq (d-1)/2$,
    \begin{equation*}
        \innerproduct{\bold{1}}{\bold{w}_r}
        = \sum_{j = 0}^{d - 1}v_{j+r}v_{j-r}
        = \sum_{k = 0}^{d - 1}v_{k+2r}v_k
        = \frac{\sigma_{2r}}{\sqrt{d+1}}, 
    \end{equation*}
    with the change of variables $k = j - r$ and where the last equality follows from \Cref{eq:Xoverlap} and  $\sigma_{-2r} = \sigma_{2r}$.
\end{proof}

Using the preceding two lemmas, we prove one last general identity.

\begin{proposition}
\label{prop:identity_related_to_one}
    Let $K \subset \R$ be a number field, and $d > 1$. If a fiducial vector $\bold{v}$ exists in $K^d$, then the following holds:
    \begin{equation*}
        1 = \frac{d + 1}{2}v_j^2 + \sqrt{d + 1}\sum_{r = 1}^{\tfrac{d - 1}{2}}\sigma_{2r}v_{j + r}v_{j - r}, \quad 0 \leq j < d.
    \end{equation*}
\end{proposition}
\begin{proof}
    Define
    \begin{equation*}
        \bold{u}
        = \frac{d+1}{2}\bold{w}_0
        + \sqrt{d+1}
        \sum_{r = 1}^{\tfrac{d-1}{2}}
        \sigma_{2r}\bold{w}_r.
    \end{equation*}
We compute
\[
\innerproduct{\bold{1}}{\bold{u}} = \frac{d+1}{2} \innerproduct{\bold{1}}{\bold{w}_0} + \sqrt{d+1} \sum_{r = 1}^{\tfrac{d-1}{2}} \sigma_{2r} \innerproduct{\bold{1}}{\bold{w}_r}.
\]
 Using \Cref{lemma:w_r_are_mutually_orthogonal} and \Cref{lemma:innerprod_of_w_r_and_all_ones}, this simplifies to
    \begin{equation*}
        \innerproduct{\bold{1}}{\bold{u}} 
        = \frac{d+1}{2} + \sqrt{d+1} \sum_{r = 1}^{\tfrac{d-1}{2}} \sigma_{2r}\frac{\sigma_{2r}}{\sqrt{d+1}} 
        = \frac{d+1}{2} + \sum_{r = 1}^{\tfrac{d-1}{2}}\sigma_{2r}^2 
        = d
    \end{equation*}
    since $\sigma_{2r}^2 = 1 $.
    Expanding the inner product of $\bold{u}$ with itself gives
    \begin{align*}
        \innerproduct{\bold{u}}{\bold{u}} &= \frac{(d+1)^2}{4} \innerproduct{\bold{w}_0}{\bold{w}_0}\\
        &+ (d+1)\sqrt{d+1} \sum_{r = 1}^{\tfrac{d-1}{2}} \sigma_{2r} \innerproduct{\bold{w}_0}{\bold{w}_r}\\
        &+ (d+1) \sum_{r = 1}^{\tfrac{d-1}{2}} \sum_{q = 1}^{\tfrac{d-1}{2}} \sigma_{2r}\sigma_{2q} \innerproduct{\bold{w}_r}{\bold{w}_q}.
    \end{align*}
    The middle sum vanishes because $\bold{w}_0$ is orthogonal to $\bold{w}_r$ for every $r>0$. In the double sum, all terms with $r\neq q$ vanish, hence only the diagonal terms remain, and
    \begin{align*}
        \innerproduct{\bold{u}}{\bold{u}} &= \frac{(d+1)^2}{4}\frac{2}{d+1} + (d+1) \sum_{r = 1}^{\tfrac{d-1}{2}} \sigma_{2r}^2 \innerproduct{\bold{w}_r}{\bold{w}_r} \\
        &= \frac{d+1}{2} + (d+1) \sum_{r = 1}^{\tfrac{d-1}{2}}\frac{1}{d+1} 
        = d.
    \end{align*}
    Since $\bold{1}$ and $\bold{u}$ are real, their inner product is symmetric. Therefore
    \begin{equation*}
        \innerproduct{\bold{1}-\bold{u}}{\bold{1}-\bold{u}} = \innerproduct{\bold{1}}{\bold{1}} - \innerproduct{\bold{1}}{\bold{u}} - \innerproduct{\bold{u}}{\bold{1}} + \innerproduct{\bold{u}}{\bold{u}}
        = 0.
    \end{equation*}
    Equivalently,
    \begin{equation*}
        \sum_{j = 0}^{d - 1}(1-u_j)^2=0.
    \end{equation*}
    Each summand is non-negative, so $u_j=1$ for every $j$. Substituting the definition of $\bold{u}$ and the coordinates of $\bold{w}_r$, we conclude that
    \begin{equation*}
        1 = \frac{d+1}{2}v_j^2 + \sqrt{d+1} \sum_{r = 1}^{\tfrac{d-1}{2}} \sigma_{2r}v_{j+r}v_{j-r}, \quad 0 \leq j \leq d-1.
    \end{equation*}
\end{proof}

%
%

\subsection{Non-Existence in Dimensions \texorpdfstring{$d \equiv 1 \pmod{4}$}{d = 1 (mod 4)}}
\label{subsec:dequiv1mod4}

We next prove that no fiducial vector can exist in dimensions $d \equiv 1 \pmod{4}$. The proof relies on the notion of prime valuation.

Let $p \in \Z$ be a prime. Recall that the $p$-adic valuation of an integer $n$ is the highest power of $p$ which divides $n$. 
More generally, let $K$ be a number field with ring of integers $\mathcal{O}_K$. Write $p\mathcal{O}_K = \prod \mathfrak{P}^{e_\Pf}$ and pick one of these primes above $p$, let us call it $\mathfrak{P}$. Then the valuation $v_{\mathfrak{P}}(n)$ is the highest power of $\Pf$ in the factorization of $n\mathcal{O}_K$.

We further recall useful identities regarding prime valuations:
\begin{align}
\label{eq:prime_valuation_identities}
    v_{\Pf}(ab) &= v_{\Pf}(a) + v_{\Pf}(b), \quad a, b \in K^{\times}, \\
    v_{\Pf}(a + b) &\geq \min\{v_{\Pf}(a), v_{\Pf}(b)\}, \quad a, b \in K^{\times}, \\
    v_{\Pf}(a + b) &= \min\{v_{\Pf}(a), v_{\Pf}(b)\}, \quad v_{\Pf}(a) \neq v_{\Pf}(b), a, b \in K^{\times},\\
    v_{\Pf}(0) &= +\infty
\end{align}
and for $\Pf$ a prime above 2,
\begin{equation}
\label{eq:Pf_valuation_in_terms_of_2_valuation}
    v_{\Pf}(q) = e_{\Pf}v_2(q), \quad q \in \Q^{\times}.
\end{equation}

\begin{proposition}
\label{prop:no_real_fiducial_in_d=1mod4}
    Let $K \subset \R$ be a number field and $d > 1$. If $d \equiv 1 \pmod{4}$, then no fiducial vector can exist in $K^d$.
\end{proposition}
\begin{proof}
    Let $d \equiv 1 \pmod{4}$, in which case $v_2(d + 1) = 1$, and suppose a fiducial vector exists in $K^d$. Let $\Pf$ be a prime above 2 in $K$.
    Using \Cref{eq:prime_valuation_identities} and \Cref{eq:Pf_valuation_in_terms_of_2_valuation}, we have:
    \begin{align*}
        v_{\Pf}(\frac{d + 1}{2}) &= e_{\Pf}v_2(\frac{d + 1}{2}) = 0, \\
        v_{\Pf}(\sqrt{d + 1}) &= \frac{v_{\Pf}(d + 1)}{2} = \frac{e_{\Pf}}{2} > 0.
    \end{align*}
    Let
    \begin{equation*}
        \lambda = \min_{0 \leq j < d} v_{\Pf}(v_j).
    \end{equation*}
    Since $\bold{v}$ is a fiducial vector, at least one coordinate is non-zero, so $\lambda$ is finite. Choose $j$ such that $v_{\Pf}(v_j) = \lambda$.
    Looking at the identity in \Cref{prop:identity_related_to_one}, we evaluate the summands as follows:
    \begin{align*}
        v_{\Pf}(\frac{d+1}{2}v_j^2) &= v_{\Pf}(\frac{d+1}{2}) + v_{\Pf}(v_j^2) = 2\lambda, \\
        v_{\Pf}(\sigma_{2r}v_{j+r}v_{j-r}) &= v_{\Pf}(\sigma_{2r}) + v_{\Pf}(v_{j + r}) + v_{\Pf}(v_{j - r}) \geq 2\lambda,
    \end{align*}
    and hence
    \begin{align*}
        v_{\Pf}(\sqrt{d+1} \sum_{r = 1}^{\tfrac{d-1}{2}} \sigma_{2r} v_{j+r}v_{j-r}) &\geq \frac{e_{\Pf}}{2} + \min\{v_{\Pf}(\sigma_{2r}v_{j+r}v_{j-r}) : 1 \leq r \leq \frac{d - 1}{2}\} \\
        &\geq \frac{e_{\Pf}}{2} + 2\lambda \\
        &> 2\lambda.
    \end{align*}
    Finally, since the valuation of the two summands are distinct, \Cref{eq:prime_valuation_identities} tells us that we have
    \begin{equation*}
        0 = v_{\Pf}(1) = v_{\Pf}((\frac{d+1}{2}v_j^2) + (\sqrt{d+1} \sum_{r = 1}^{\tfrac{d-1}{2}} \sigma_{2r}v_{j+r}v_{j-r})) = 2\lambda.
    \end{equation*}
    This implies that $\lambda = 0$ and since it is a minimum by definition, the valuations at $\Pf$ of the coordinates are all non-negative. Recall \Cref{eq:Xoverlap}; let $a\neq 0$, we have
    \begin{equation*}
        0 \leq 2\lambda \leq v_{\Pf}(\sum_{j = 0}^{d - 1} v_jv_{j - a}) = v_{\Pf}(c_a) = v_{\Pf}(\frac{\sigma_a}{\sqrt{d + 1}}) = -v_{\Pf}(\sqrt{d + 1}) = - \frac{e_{\Pf}}{2} < 0,
    \end{equation*}
    a contradiction.
\end{proof}

If a real fiducial vector exists in $K^d$, then \Cref{theorem:no_real_fiducial_in_even_dimensions} asserts that $d$ must be odd, and since \Cref{prop:no_real_fiducial_in_d=1mod4} says $d \equiv 1 \pmod{4}$ is impossible, we must have $d\equiv 3 \pmod{4}$.

\begin{theorem}
\label{theorem:real_fiducial_must_be_in_d=3mod4}
     Let $K \subset \R$ be a number field and $d > 1$. If a fiducial vector $\bold{v}$ exists in $K^d$, then $d \equiv 3 \pmod{4}$.
\end{theorem}


%
%
%

\section{Totally Real Extensions of \texorpdfstring{$\Q$}{Q}}
\label{sec:totallyreal}

Recall that an algebraic number $\alpha \in \Qbar$ is called \textit{totally real} if its minimal polynomial has distinct real roots, all of which are real. A number field $K$ is called \textit{totally real} if for any embedding $\sigma : K \hookrightarrow \C$ of $K$, its image is in $\R$, that is, $\sigma(K) \subset \R$. Let $\Q^{tr}$ be the field of totally real algebraic numbers. This gives rise to a hierarchy of field families: 
\begin{align*}
    \Qbar &= \bigcup \Q(\alpha), \quad \alpha \in \Qbar, \\
    \Ralg &= \bigcup \Q(\alpha), \quad \alpha \in \Ralg, \\
    \Qtr &= \bigcup \Q(\alpha), \quad \text{$\alpha \in \Qbar$ is totally real,} \\ 
    \Q &\subsetneq \Q^{tr} \subsetneq \Ralg \subsetneq \Qbar.
\end{align*}
Inside $\Q^{tr}$ is found the family $\Q(\omega_n)^+$, $n \geq 4$, of maximal totally real subfields of the cyclotomic fields $\Q(\omega_n)$ for $\omega_n$ a primitive $n$th root of unity. By the Kronecker-Weber Theorem \cite[Theorem 4.16]{milneCFT}, this family contains all totally real fields whose Galois group is Abelian.

We will prove that if $K$ is a totally real number field, then $d=4m^2-1$, for $m\in\Z_{\geq 1}$. 
We first prove a useful lemma.

\begin{lemma}
\label{lemma:Hsquared_identities}
    For $d \equiv 3 \pmod{4}$, let $H$ be a $d \times d$ symmetric circulant matrix defined by
    \begin{equation*}
        H_{jl} = \begin{cases}
            0, & j = l, \\
            \sigma_{l - j}, & j \neq l,
        \end{cases}
    \end{equation*}
    where $\sigma_a = \pm 1$ for $1 \leq a < d$. Then, we have the following relations:
    \begin{align*}
        (H^2)_{jj} &= d - 1, \quad 0 \leq j < d, \\
        (H^2)_{jl} &\equiv 1 \pmod{4}, \quad j \neq l.
    \end{align*}
\end{lemma}
\begin{proof}
    We start with the diagonal entries. Since $H$ is symmetric, we have
    \begin{equation*}
        (H^2)_{jj}
        = \sum_{k = 0}^{d - 1} H_{jk}H_{kj} 
        = \sum_{k = 0}^{d - 1} H_{jk}^2 
        = \sum_{\substack{k = 0 \\ k \neq j}}^{d - 1} \sigma_{k - j}^2 
        = d - 1,
    \end{equation*}
    where we used $H_{jj} = 0$ and $\sigma_{k - j}^2 = 1$ for $k \neq j$.

    For the off-diagonal entries, fix distinct indices $j,l$. By the symmetry of $H$,
    \begin{equation*}
        (H^2)_{jl}
        = \sum_{k = 0}^{d - 1} H_{jk}H_{kl} 
        = \sum_{k = 0}^{d - 1} H_{jk}H_{lk} 
        = \sum_{\substack{k = 0 \\ k \neq j,l}}^{d - 1} H_{jk}H_{lk}.
    \end{equation*}
    The last equality follows because the term with $k = j$ contains $H_{jj} = 0$, and the term with $k = l$ contains $H_{ll} = 0$.

    We compare rows $j$ and $l$ after deleting columns $j$ and $l$ from both rows. Each shortened row consists of $d - 2$ entries, all equal to $1$ or $-1$.

    Since $H$ is circulant, its rows are cyclic shifts of one another and therefore contain the same number of negative entries. Deleting columns $j$ and $l$ preserves this equality, since from row $j$ we delete $H_{jj} = 0 $ and $H_{jl}$, 
    while from row $l$ we delete $H_{lj}$ and $H_{ll} = 0$.
    Since $H_{jl} = H_{lj}$, we remove one zero and the same sign from each row. Hence, the shortened rows still contain the same number of negative entries.

    For each remaining column $k$, there are four possibilities for the pair $(H_{jk},H_{lk})$. Let $a,p,q,b$ denote their respective numbers of occurrences:
    \begin{equation*}
        \begin{array}{c|c|c}
            (H_{jk},H_{lk}) & \text{Number of occurrences} & H_{jk}H_{lk} \\
            \hline
            (1,1)   & a & 1 \\
            (1,-1)  & p & -1 \\
            (-1,1)  & q & -1 \\
            (-1,-1) & b & 1
        \end{array}
    \end{equation*}
    Since there are $d - 2$ remaining columns, we have
    \begin{equation*}
        a + p + q + b = d - 2.
    \end{equation*}

    The shortened row $j$ contains $q + b$ negative entries, whereas the shortened row $l$ contains $p + b$ negative entries. These numbers are equal, so
    \begin{equation*}
        q + b = p + b,
    \end{equation*}
    and therefore $p = q$. In particular, the two shortened rows disagree in exactly $p + q = 2p$ positions.

    Summing the products in the table now gives
    \begin{equation*}
        (H^2)_{jl}
        = a + b - p - q 
        = d - 2 - 2(p + q) 
        = d - 2 - 4p.
    \end{equation*}
    Since $d \equiv 3 \pmod{4}$, it follows that
    \begin{equation*}
        (H^2)_{jl} \equiv d - 2 \equiv 1 \pmod{4},
    \end{equation*}
    as required.
\end{proof}

We will use the above lemma to prove the following proposition.

\begin{proposition}
\label{prop:vec_satisfying_Xoverlap_in_3mo4_implies_real_embedding_fixes_sqrtd1}
    Let $K \subset \R$ be a number field, $d > 3$, and $d \equiv 3 \pmod{4}$. If a normalized vector $\bold{v} \in K^d$ satisfies
    \begin{equation*}
        \sum_{j = 0}^{d - 1} v_{j - a}v_j = \frac{\sigma_a}{\sqrt{d + 1}}, \quad \sigma_a = \pm 1, 1 \leq a < d,
    \end{equation*}
    then every real embedding $\tau : K \hookrightarrow \R$ fixes $\sqrt{d + 1}$.
    In particular, if $\sqrt{d + 1} \not\in \Q$, then $K$ is not totally real and $4 \mid [K:\Q]$.
\end{proposition}
\begin{proof}
    Recall that $\sigma_a = \sigma_{-a}$, where the indices are taken modulo $d$.
    Let $H$ be the symmetric circulant matrix defined by
    \begin{equation*}
        H_{jl} = \begin{cases}
            0, & j = l, \\
            \sigma_{l - j}, & j \neq l,
        \end{cases}
    \end{equation*}
    which satisfies the conditions for \Cref{lemma:Hsquared_identities}.
    Then the Gram matrix $G$ of $\{\bold{v}_{a, 0} : 0 \leq a \leq d - 1\}$ is defined by
    \begin{equation*}
        G_{jl} = 
        \begin{cases}
            \innerproduct{\bold{v}_{j, 0}}{\bold{v}_{l, 0}} = \sum_{k = 0}^{d - 1}v_{k - j}v_{k - l} = \sum_{x = 0}^{d - 1} v_{x - (j - l)}v_{x} = \frac{\sigma_{j - l}}{\sqrt{d + 1}}, & j \neq l \\
            1, & j=l,
        \end{cases}
    \end{equation*}
    which may thus be rewritten as
    \begin{equation*}
        G = I_d + \frac{H}{\sqrt{d + 1}}.
    \end{equation*}
    \Cref{lemma:dimensions_4m2-1} tells us that $\sqrt{d + 1} \in K$.  
    
    Suppose that there exists a real embedding $\tau : K \hookrightarrow \R$ such that $\tau(\sqrt{d + 1}) = -\sqrt{d + 1}$. Let $\bold{u} = \tau(\bold{v})$ be the vector where $\tau$ is applied to each entry of $\bold{v}$. This preserves normalization ($\langle \bold{u}, \bold{u} \rangle =1$) and the overlap signs are reversed ($\tau(\tfrac{\sigma_{j - l}}{\sqrt{d + 1}})=\tfrac{\sigma_{j - l}}{-\sqrt{d + 1}}$); therefore, the Gram matrix associated with the cyclic shifts of $\bold{u}$ is
    \begin{equation*}
        G' = I_d - \frac{H}{\sqrt{d + 1}}.
    \end{equation*}
    Observe that $G$ and $G'$ are both positive semidefinite since they are Gram matrices, and hence their eigenvalues are real and non-negative.
    Let $\lambda$ be an eigenvalue of $H$. Then the corresponding eigenvalues of $G$ and $G'$ are $1 + \frac{\lambda}{\sqrt{d + 1}}$ and $1 - \frac{\lambda}{\sqrt{d + 1}}$, respectively, and their non-negativity implies that $\lambda^2 \leq d + 1$.
    
    Now define
    \begin{equation*}
        B = (d + 1)I_d - H^2.
    \end{equation*}
    The eigenvalues of $B$ are $d + 1 - \lambda^2 \geq 0$, hence $B$ is also positive semidefinite. We are going to prove this is impossible.

    We first compute the entries of $B$.
    Its diagonal entries are $B_{jj} = d + 1 - (H^2)_{jj} = d + 1 - (d - 1) = 2$ by \Cref{lemma:Hsquared_identities}.

    Next, we compute the non-diagonal entries. Let $j \neq l$. Since $B$ is positive semidefinite, its principal minors are always non-negative, that is,
    \begin{equation*}
        0 \leq \det\begin{pmatrix}
            2 & B_{jl} \\
            B_{jl} & 2
        \end{pmatrix} = 4 - B_{jl}^2,
    \end{equation*}
    and since $B_{jl} = - (H^2)_{jl}$, we have $|(H^2)_{jl}| = |B_{jl}| \leq 2$ for all $j \neq l$. Together with \Cref{lemma:Hsquared_identities}, which says that $(H^2)_{jl} \equiv 1 \pmod{4}$, this implies $(H^2)_{jl} = 1$, and hence $B_{jl} = - 1$. Since every off-diagonal entry appears in a principal minor, this forces $B=3I_d-\bold{1}_{d\times d}$ for $\bold{1}_{d\times d}$ the all-ones $d\times d$ matrix.

    Finally, take the all-ones vector $\bold{1}$. We have for $d > 3$
    \begin{equation*}
        \bold{1}^T B \bold{1} = \bold{1}^T ((3 - d) \bold{1}) = d(3 - d) < 0.
    \end{equation*}
    This contradicts the fact that $B$ is positive semidefinite, and therefore every real embedding of $K$ fixes $\sqrt{d + 1}$.

    Furthermore, suppose $\sqrt{d + 1} \not\in \Q$, and let $F_d = \Q(\sqrt{d + 1})$. Then the embedding of $F_d$ sending $\sqrt{d + 1}$ to $-\sqrt{d + 1}$ extends to exactly $[K:F_d]$ embeddings of $K$ into $\C$. By the preceding argument, none of these embeddings are real, and hence they occur in complex-conjugate pairs. Consequently, $K$ is not totally real, and $2 \mid [K:F_d]$. Since $[K:\Q] = 2[K:F_d]$, this implies $4 \mid [K:\Q]$.
\end{proof}

This is a general statement on normalized vectors. Real fiducial vectors satisfy the hypothesis, which is the $\bold{X}$-overlap equation stated in \Cref{eq:Xoverlap}. Moreover we have shown that a real fiducial vector in $K^d$ can only exist if $d \equiv 3 \pmod{4}$ in \Cref{theorem:real_fiducial_must_be_in_d=3mod4}. We therefore have the following theorem:

\begin{theorem}
\label{theorem:K_totally_real_implies_sqrtd1_in_Q}
    Let $K \subset \R$ be a totally real number field, and $d > 1$. If a fiducial vector $\bold{v}$ exists in $K^d$, then $\sqrt{d + 1} \in \Q$.
\end{theorem}
\begin{proof}
    If $d = 3$, then $\sqrt{3 + 1} = 2 \in \Q$, so let $d > 3$. Suppose $\sqrt{d + 1} \notin \Q$. Then the embedding
    \begin{equation*}
        \Q(\sqrt{d + 1}) \hookrightarrow \R, \quad \sqrt{d + 1} \mapsto -\sqrt{d + 1}
    \end{equation*}
    extends to an embedding of $K$ into $\R$ since $K$ is totally real, but this contradicts \Cref{prop:vec_satisfying_Xoverlap_in_3mo4_implies_real_embedding_fixes_sqrtd1}. 
\end{proof}

By \Cref{lemma:dimensions_4m2-1}, this restricts the possible dimensions to $d = 4m^2 - 1, m \in \Z_{\geq 1}$ for $K$ totally real.

If \Cref{conj:real_fid_vec_in_4m2+3} were to be true, that is, $d = 4m^2 + 3, m \in \Z_{\geq 0}$, there would be an interesting conclusion.

\begin{corollary}
    Let $K \subset \R$ be a number field, and $d = 4m^2 + 3, m \in \Z_{\geq 1}$. Suppose a fiducial vector $\bold{v}$ exists in $K^d$. Then $K$ cannot be a totally real extension of $\Q$. In particular, $\bold{v}$ cannot exist in $K^d$ for $K$ a real Abelian extension of $\Q$.
\end{corollary}
\begin{proof}
    Suppose $K$ were a totally real number field. By \Cref{theorem:K_totally_real_implies_sqrtd1_in_Q} we must have $\sqrt{d + 1} = 2\sqrt{m^2 + 1} \in \Q$, but for $m \geq 1$ we have $m^2 < m^2 + 1 < (m + 1)^2$, hence $m^2 + 1$ is not a square, a contradiction. 
    
    Furthermore, $\Q(\omega_n)^+$ is totally real, and since any real Abelian extension is a subextension of some $\Q(\omega_n)^+$ by the Kronecker-Weber theorem \cite[Theorem 4.16]{milneCFT}, we have the desired conclusion.
\end{proof}

\section{Field Specific Conclusions}
\label{sec:fieldspecific}

We know from \Cref{theorem:real_fiducial_must_be_in_d=3mod4} that if there exists a fiducial vector $\bold{v} \in K^d$, $K \subset \R$, then $d \equiv 3 \pmod {4}$. In order to gain further knowledge about $d$, we will focus on specific types of number fields. We begin by exploiting tame ramification above 2, and then distinguish even and odd degree extensions of $\Q$.

Recall that \Cref{lemma:dimensions_4m2-1} tells us that
\begin{equation*}
     \Q(\sqrt{d + 1}) \subseteq K.
\end{equation*}
If $\sqrt{d + 1} \not\in \Q$, then $K$ has a quadratic subfield, and hence, $K$ must be an even degree extension of $\Q$. If $\sqrt{d + 1} \in \Q$, then we are restricted to the family $d = 4m^2 - 1, m \in \Z_{\geq 1}$. 

%
%
%

\subsection{Real Fields with Tame Ramification above 2}

A prime $\Pf$ over $p$ is called \textit{wild} if $p \mid e_{\Pf \mid p}$, and \textit{tame} if $p \nmid e_{\Pf \mid p}$. In particular, if we take a prime $\Pf$ above $2$, then $e_{\Pf \mid 2}$ is even if and only if it is wild, and similarly $e_{\Pf \mid 2}$ is odd if and only if it is tame; note that this includes the unramified case, that is, $e_{\Pf \mid 2} = 1$.

In this subsection, we investigate what can be said about the $2$-adic valuation of $d+1$, depending on the behaviour of a prime $\Pf$ above 2.

\begin{proposition}
\label{prop:odd_ramification_index_implies_v2d+1equiv2mod4}
    Let $K \subset \R$ be a number field, and $d > 1$. Suppose some prime ideal $\Pf$ above $2$ in $K$ has an odd ramification index $e_{\Pf}$. If a fiducial vector $\bold{v}$ exists in $K^d$, then
    \begin{equation*}
        v_2(d + 1) \equiv 2 \pmod{4}.
    \end{equation*}
\end{proposition}
\begin{proof}
    Suppose that $\bold{v} \in K^d$ is a fiducial vector.
    Since $K \subset \R$, \Cref{theorem:no_real_fiducial_in_even_dimensions}
    implies that $d$ is odd. Recall that \Cref{lemma:dimensions_4m2-1} tells us that $\sqrt{d + 1} \in K$. Therefore using \Cref{eq:prime_valuation_identities} we have,
    \begin{equation*}
        2v_{\Pf}(\sqrt{d + 1})
        = v_{\Pf}(d + 1)
        = e_{\Pf}v_2(d + 1).
    \end{equation*}
    Since $e_{\Pf}$ is odd, $v_2(d + 1)\geq 0$ must be even.
    Moreover, $d$ is odd, so $v_2(d + 1) \neq 0$, and consequently $v_2(d + 1) \geq 2$. 
    
    If $v_2(d + 1) = 2$, the desired congruence already holds.

    Suppose therefore that $v_2(d + 1) > 2$.
    In particular, $4 \mid d + 1$.
    Since $4 \mid d + 1$, the integer $\tfrac{d - 1}{2}=\tfrac{d + 1}{2}-1$ is odd.
    By Lemma \ref{lem:Sfodd},
    \begin{equation*}
        \Sf \equiv \frac{d - 1}{2} \equiv 1 \pmod{2},
    \end{equation*}
    so $\Sf$ is odd.
    By \Cref{lemma:Tf_equation}, we have
    \begin{equation*}
        \Tf^2 = 1 + \frac{2\Sf}{\sqrt{d + 1}}
    \end{equation*}
    and $v_{\Pf}(\Sf) = e_{\Pf}v_2(\Sf) = 0$,
    and hence
    \begin{align*}
        v_{\Pf}(\frac{2\Sf}{\sqrt{d + 1}})
        &= v_{\Pf}(2) + v_{\Pf}(\Sf)
           - v_{\Pf}(\sqrt{d + 1}) \\
        &= e_{\Pf} - \frac{e_{\Pf}v_2(d + 1)}{2} \\
        &= e_{\Pf}\left(1 - \frac{v_2(d + 1)}{2}\right) < 0.
    \end{align*}
    Since $v_{\Pf}(1) = 0$, the two summands in the expression
    for $\Tf^2$ have distinct valuations. The valuation of their
    sum is therefore the minimum of their valuations.
    In particular, $\Tf \neq 0$, and
    \begin{equation}
    \label{eq:valuation_of_Tfsquared}
        2v_{\Pf}(\Tf)
        = v_{\Pf}(1 + \frac{2\Sf}{\sqrt{d + 1}})
        = e_{\Pf}(1 - \frac{v_2(d + 1)}{2}).
    \end{equation}
    The left-hand side is even. Since $e_{\Pf}$ is odd,
    $1 - \tfrac{v_2(d + 1)}{2}$ must be even, and hence $\tfrac{v_2(d + 1)}{2}$
    is odd. Therefore,
    \begin{equation*}
        v_2(d + 1) \equiv 2 \pmod{4},
    \end{equation*}
    as required.
\end{proof}

\begin{corollary}
\label{cor:odd_ramification_index+d=4m2-1}
    Let $K \subset \R$ be a number field, and $d > 1$. If a fiducial vector $\bold{v}$ exists in $K^d$, and if some prime ideal $\Pf$ above $2$ is tame, then in dimensions $d = 4m^2 - 1$, $m\in\Z_{\geq 1}$, $v_2(m)$ must be even.
\end{corollary}
\begin{proof}
    Substituting $d = 4m^2 - 1, m \in \Z_{\geq 1}$ into \Cref{prop:odd_ramification_index_implies_v2d+1equiv2mod4}, we get
    \begin{equation*}
        v_2(d + 1) = v_2(4m^2) = 2 + 2v_2(m) \equiv 2 \pmod{4} \iff v_2(m) \text{ is even},
    \end{equation*}
    as desired.
\end{proof}

%
%
%

\subsection{Real Number Fields of Odd Degree}
\label{subsec:odd_degree_K}

If $K$ is an odd degree extension of $\Q$, then by \Cref{lemma:dimensions_4m2-1}, $\sqrt{d + 1} \in \Q$, and we are restricted to dimensions $d = 4m^2 - 1, m \in \Z_{\geq 1}$. We can further restrict this using \Cref{cor:odd_ramification_index+d=4m2-1}.



\begin{proposition}
\label{prop:degree_of_K_is_odd_implies_v2m_is_even}
    Let $K \subset \R$ be a number field such that $[K:\Q]$ is odd. Then no fiducial vector can exist in $K^d$ unless $d = 4m^2 - 1$ where $m \in \Z_{\geq 1}$ and $v_2(m)$ is even.    
\end{proposition}
\begin{proof}
    \Cref{lemma:dimensions_4m2-1} forces $d = 4m^2 - 1, m \in \Z_{\geq 1}$. Consider the factorization of $2$ in $K$:
    \begin{equation*}
        2\OO_K = \prod_{\pf \mid 2} \pf^{e_{\pf}}.
    \end{equation*}
    Since $[K:\Q]$ is odd,  
    the degree formula \cite[Theorem 3.34]{milneANT}, 
    \begin{equation*}
        [K:\Q] = \sum_{\pf \mid 2} e_{\pf}f_{\pf}, \quad f_{\pf} = [\OO_K/\pf : \F_2],
    \end{equation*}
    tells us that at least one summand is odd, that is, there is an odd $e_{\pf}$, and hence by \Cref{cor:odd_ramification_index+d=4m2-1} we get the desired result.
\end{proof}

Since the above result is true for an arbitrary real number field, it is true in particular for the maximal totally real subfield of a cyclotomic field, as long as its degree is odd. 

\begin{proposition}
\label{prop:real_cyclo_field_is_odd_ext}
    Let $K = \Q(\omega_n)^+$ where $n \geq 3$. Then $[K:\Q]$ is odd if and only if $n = 4, p^e, 2p^e$ where $p$ is a prime such that $p \equiv 3 \pmod{4}$, and $e \geq 1$.
\end{proposition}
\begin{proof}
    We know $[\Q(\omega_n):\Q] = 2[K:\Q] = \phi(n)$, where $\phi(n)$ is Euler's Totient function. Hence, $[K:\Q]$ is odd if and only if $4 \nmid \phi(n)$.
    If $n = \prod_{p\mid n} p^{e_p}$, then
    \begin{equation*}
        \phi(n) = \prod_{p\mid n} p^{e_p - 1}(p - 1).
    \end{equation*}
    If $n$ has two distinct prime factors, $p, q > 2$, then $\phi(n)$ is divisible by $4$, since $(p - 1)(q - 1) \mid \phi(n)$. Therefore $n$ can only have one prime factor, or two prime factors, one of which must be 2.
    \begin{enumerate}
        \item One prime factor case: Let $n = p^e, e \geq 1$. 
        \begin{enumerate}
            \item $p = 2$: If $e \geq 3$, then $4 \mid \phi(n) = 2^{e - 1}$, hence $e \leq 2$ and therefore we are left with $n = 4$ for which $\phi(4) = 2$ which is not divisible by 4.
            \item $p \equiv 1 \pmod{4}$: In this case $\phi(n) = p^{e - 1} (p - 1) \equiv 0 \pmod{4}$, hence $4 \mid \phi(n)$.
            \item $p \equiv 3 \pmod{4}$: In this case $\phi(n) = p^{e - 1} (p - 1) \not \equiv 0 \pmod{4}$, hence $4 \nmid \phi(n)$.
        \end{enumerate}
        \item Two prime factor case: $n = 2^f p^e, e \geq 1$ and $1 \leq f \leq 2$.
        \begin{enumerate}
            \item $f = 2$: In this case $\phi(n) = 2p^{e - 1}(p - 1) \equiv 0 \pmod{4}$.
            \item $f = 1$: In this case $\phi(n) = \phi(n/2) = \phi(p^e)$ and hence we must have $p \equiv 3 \pmod{4}$.
        \end{enumerate}
    \end{enumerate}
    Hence, the only $n$ for which $4 \nmid \phi(n)$ are as stated in the proposition.
\end{proof}

\subsection{Real Number Fields of Even Degree}
\label{subsec:real_even_ext}

Let $K$ be a number field such that $[K:\Q]$ is even. We will treat four scenarios: 
\begin{enumerate}
    \item $[K:\Q]\equiv 2 \pmod{4}$.
    \item $K = \Q(\sqrt{D})$ where $D > 1$, and $D$ is squarefree. 
    \item $K = \Q(\omega_n)^+$.
    \item $K = \Q(b)$ corresponding to $d=7$.
\end{enumerate}

Firstly, \Cref{prop:vec_satisfying_Xoverlap_in_3mo4_implies_real_embedding_fixes_sqrtd1} shows that the case when $4 \mid [K:\Q]$ is special: if $\sqrt{d + 1} \notin \Q$, then $K \not\subset \Q^{tr}$, and $4\mid [K:\Q]$.

We discuss the case when $4\nmid [K:\Q]$, that is, $[K:\Q] \equiv 2 \pmod{4}$.

\begin{proposition}
\label{prop:even_degree_not_div_by_4_implies_sqrtd1_in_Q}
    Let $K \subset \R$ be a number field, and $[K:\Q] \equiv 2 \pmod{4}$. If a fiducial vector exists in $K^d$, where $d > 3$, we are restricted to $d = 4m^2 - 1, m \in \Z_{\geq 1}$.
\end{proposition}
\begin{proof}
    Suppose that $\sqrt{d + 1} \not\in \Q$. Then \Cref{theorem:real_fiducial_must_be_in_d=3mod4} tells us that real fiducial vectors can only exist in dimensions $d \equiv 3 \pmod{4}$, hence the hypotheses of \Cref{prop:vec_satisfying_Xoverlap_in_3mo4_implies_real_embedding_fixes_sqrtd1} are satisfied, and therefore $4 \mid [K:\Q]$, a contradiction.
\end{proof}

Next, we consider the real quadratic extensions of $\Q$ and real cyclotomic fields of even degree. Note that both of these cases are examples of totally real fields, and hence \Cref{theorem:K_totally_real_implies_sqrtd1_in_Q} implies that $\sqrt{d + 1} \in \Q$.


Firstly, we recall the discriminant of $\Q(\sqrt{D})$ \cite[Exercise 2.4]{neukirch2010algebraic}, where $D$ is squarefree.
\begin{equation*}
    \Delta = \begin{cases}
        D, \quad D \equiv 1 \pmod{4} \\
        4D, \quad D \equiv 2, 3 \pmod{4}
    \end{cases}
\end{equation*}

\begin{proposition}
\label{prop:quad_ext_d=1mod4}
    Let $K = \Q(\sqrt{D})$ where $D \equiv 1 \pmod{4}$. If a fiducial vector exists in $K^d$, we are restricted to $d = 4m^2 - 1, m \in \Z_{\geq 1}$, and $v_2(m)$ is even.
\end{proposition}
\begin{proof}
    The number field $\Q(\sqrt{D})$ is totally real, which by \Cref{theorem:K_totally_real_implies_sqrtd1_in_Q} implies $\sqrt{d + 1} \in \Q$, that is $d = 4m^2 - 1, m \in \Z_{\geq 1}$. Since the discriminant of $K$ is $D$, which is odd, $2$ is unramified in $K$, we get $v_2(m)$ is even from \Cref{cor:odd_ramification_index+d=4m2-1}.
\end{proof}

\begin{proposition}
\label{prop:quad_ext_D_odd_and_d+1_square_in_Q}
    Let $K = \Q(\sqrt{D})$, where $D$ is odd. If a fiducial vector exists in $K^d$, we are restricted to $d = 4m^2 - 1, m \in \Z_{\geq 1}$, and $v_2(m)$ is even.
\end{proposition}
\begin{proof}
    The number field $\Q(\sqrt{D})$ is totally real, which by \Cref{theorem:K_totally_real_implies_sqrtd1_in_Q} implies $\sqrt{d + 1} \in \Q$, that is, $d = 4m^2 - 1, m \in \Z_{\geq 1}$. Therefore, \Cref{lemma:dimensions_4m2-1} together with \Cref{lemma:Tf_equation} tells us that
    \begin{equation*}
        \Tf^2  = 1 + \frac{2 \Sf}{\sqrt{d+1}} = \frac{m + \Sf}{m} \in \Q.
    \end{equation*}
    Let $\Tf = x + y\sqrt{D}$ where $x, y \in \Q$. Since $\Tf^2 \in \Q$ we have
    \begin{equation*}
        x^2 + Dy^2 + 2xy\sqrt{D} \in \Q,
    \end{equation*}
    hence $xy = 0$, and therefore either $x^2 = \frac{m + \Sf}{m}$ or $Dy^2 = \frac{m + \Sf}{m}$. 
    
    If $D$ is odd, then both $v_2(x^2) = 2v_2(x)$ and $v_2(Dy^2) = v_2(D) + 2v_2(y) = 2v_2(y)$ are even, hence $v_2(\frac{m + \Sf}{m}) = v_2(m + \Sf) - v_2(m)$ is even. 
    
    Then, if $v_2(m)$ is odd, then $m$ is even, which implies $v_2(m + \Sf) = 0$ and we get a contradiction. Hence we must have $v_2(m)$ being even.
\end{proof}

We remark that we recover \Cref{prop:quad_ext_d=1mod4} without using tame ramification above $2$.

\begin{proposition}
\label{prop:real_cyclo_field_even_ext_4_does_not_divide_n}
    Let $K = \Q(\omega_n)^+$, where $n \geq 3$, and suppose $[K:\Q]$ is even, equivalently $4 \mid \phi(n)$, and $4 \nmid n$. If a fiducial vector exists in $K^d$, we are restricted to $d = 4m^2 - 1, m \in \Z_{\geq 1}$, and $v_2(m)$ is even.
\end{proposition}
\begin{proof}
    The number field $\Q(\omega_n)^+$ is totally real, which by \Cref{theorem:K_totally_real_implies_sqrtd1_in_Q} implies $\sqrt{d + 1} \in \Q$, that is, $d = 4m^2 - 1, m \in \Z_{\geq 1}$. Furthermore, if $n$ is odd, we know that only primes dividing $n$ can ramify in $\Q(\omega_n)$ \cite[Theorem 6.4c]{milneANT}. 
     
     If $n$ is odd, 2 is unramified in $\Q(\omega_n)$ and therefore also in $K$ (since ramification indices multiply in the tower $\Q \subset K \subset \Q(\omega_n)$), and hence by \Cref{cor:odd_ramification_index+d=4m2-1} we conclude that $v_2(m)$ is even.
     
     If $n \equiv 2 \pmod{4}$, then $n = 2s$ where $s$ is odd. Since $K = \Q(\omega_{2s})^+ = \Q(\omega_s)^+$, the first case applies again. 
\end{proof}

%
%


We study a concrete example of a number field of even degree. It is the field corresponding to the real fiducial vector of \Cref{example:d7_vector}, for 
$d = 7$.
\begin{example}
    We showed in \Cref{example:d7_fiducial_and_sic_field} that the fiducial vector from \Cref{example:d7_vector} exists in $K^7$ for $K = \Q(b)$, a degree 8 extension of $\Q$, which is not totally real, where
    \begin{equation*}
        b = \frac{1}{2}\sqrt{\frac{10 - \sqrt{2} + \sqrt{14 + 28\sqrt{2}}}{14}}.
    \end{equation*}
    Furthermore, we showed that $\Q(\sqrt{2}) \subseteq \Q(b)$ since
    \begin{equation*}
        \sqrt{2} = \frac{12 - (56b^2 - 10)^2}{112b^2 - 48}.
    \end{equation*}
    Suppose we wanted to determine for which dimensions $d = 2^a b^2 - 1$, where $a, b \in \Z_{\geq 0}$, and $2\nmid b$, a fiducial vector could exist in $K^d$.

    Since $[K:\Q] = 8$ is divisible by four, and $K$ is not totally real, we cannot claim that we are restricted to the family $d = 4m^2 - 1, m \in \Z_{\geq 1}$. Hence we need to split our analysis into two cases:
    
    \begin{enumerate}
        \item We have $\Q(\sqrt{d+1})=\Q$, that is, $d = 2^ab^2 - 1$ where $a$ is even and $b$ is odd.
        \item We have $d$ such that $\Q(\sqrt{d+1}) = \Q(\sqrt{2})$, in which case $d = 2^ab^2 - 1$ for $a,b$ odd.
    \end{enumerate}
    Furthermore, we cannot use \Cref{prop:odd_ramification_index_implies_v2d+1equiv2mod4} to restrict $d$ since it can be computed that
    \begin{equation*}
        2\OO_K = \Pf^4, \quad \Pf = \sqrt{2 - \sqrt{2}}\OO_K,
    \end{equation*}
    which shows that there is no prime above $2$ with tame ramification.
    
    We investigate the cases as follows:
    \begin{enumerate}
        \item If $a$ is even, then $\sqrt{d + 1} = 2^{a/2}b \in \Q$, hence, we are restricted to $d = 4m^2 - 1, m\in \Z_{\geq 1}$. We can find a real fiducial vector in the special case of $d = 3$, the only element belonging to $\{4m^2 - 1,~m \in\Z_{\geq 1}\} \cap \{4n^2 + 3,~n\in\Z_{\geq 0}\}$. Indeed, $4m^2-1=4n^2+3 \iff (m-n)(m+n)=1$ which gives $n=0,m=1$ that is $d=3$. 
        A real fiducial vector is given in \Cref{ex:d=3}:
        \begin{equation*}
            \frac{1}{\sqrt{2}} \begin{pmatrix}
                0 & 1 & - 1
            \end{pmatrix}^T \in \Q(\sqrt{2})^3 \subset \Q(b)^3.
        \end{equation*}
        \item If $a$ is odd, then $\sqrt{2} \in K\backslash \Q$, in which case we know from \Cref{prop:vec_satisfying_Xoverlap_in_3mo4_implies_real_embedding_fixes_sqrtd1} that we must have $4\mid [K:\Q]$, and $K$ is not totally real, which is compatible with $K = \Q(b)$, hence, we cannot rule out non-existence of a fiducial vector in $K$. We can find a real fiducial vector when $a = 3, b = 1$, that is when $d = 7$.
    \end{enumerate}
\end{example}

%
%
%

\section{Conditional Results}
\label{sec:conditional_results}

In this section, we prove some additional results, conditional on various statements and conjectures found in the literature.

%
%
%

\subsection{Assuming All \texorpdfstring{$\sigma_a$}{sigmaa} Are Positive}

The all-positive sign condition arises naturally in constructions of almost flat SIC fiducials. In \cite[Theorem 16]{Appleby_2022_StarkUnit}, a real fiducial with $\sigma_a = 1$ for every $a \ne 0$ is shown to give an almost flat fiducial under Fourier transformation. Furthermore, \cite[Equation 6]{bengtsson2025conjectureflatsicpovms} shows that the almost flat ansatz has positive $\bold{Z}$-overlaps, which correspond to positive $\bold{X}$-overlaps for its real Fourier counterpart.

This motivates the next result.

\begin{proposition}
    Let $K \subset \R$ be a number field, and $d > 1$. Suppose $\sigma_a = 1$ for all $1 \leq a \leq \frac{d - 1}{2}$. Then no fiducial vector can exist if $[K:\Q]$ is odd.
\end{proposition}
\begin{proof}
    Suppose $[K:\Q]$ is odd. By \Cref{subsec:odd_degree_K}, $d = 4m^2 - 1, m \in \Z_{\geq 1}$. \Cref{lemma:Tf_equation}, in the case where $\sigma_a = 1$ for all $1 \leq a \leq \frac{d - 1}{2}$, recalling that $\sigma_a = \sigma_{-a}$, yields
    \begin{equation*}
        \Tf^2 = 1 + \frac{d - 1}{2m} = \frac{2m + 4m^2 - 2}{2m} = \frac{(m + 1)(2m - 1)}{m} \in \Q.
    \end{equation*}
    Since $\Tf \in K$, and $\Tf^2 \in \Q$, we must have that $\Tf \in \Q$; otherwise $\Q(\Tf)$ would be a quadratic subfield of $K$. Hence $\frac{(m + 1)(2m - 1)}{m}$ must be a rational square in $\Q$. However,
    \begin{equation*}
        \gcd(m, (m + 1)(2m - 1)) = \gcd(m, (m + 1)(2m - 1)\!\!\!\pmod{m}) = \gcd(m, -1) = 1,
    \end{equation*}
    hence both $(m + 1)(2m - 1)$ and $m$ must be rational squares. However, this is incompatible modulo 8 since $m$ being a square forces $m \equiv 0, 1, 4 \pmod{8}$, but in these cases $(m + 1)(2m - 1) \equiv 7, 2, 3 \pmod{8}$ respectively, which means $(m + 1)(2m - 1)$ can never be a square.
\end{proof}

%
%
%

\subsection{Assuming the Ray Class Field Conjecture}
\label{subsec:Kd}

We recall \Cref{conj:ray_class_field_conj} for a real fiducial vector $\bold v$.

Let $\Pi(\bold{v})=\bold v\bold v^T$, then the corresponding SIC field is defined by
\begin{equation*}
    E_{\Pi(\bold v)}=\Q(\omega_{2d},\Pi(\bold{v})_{jl}:0\leq j,l<d).    
\end{equation*}
The conjecture states that 
\begin{equation}
    \K_d=\Q(\sqrt{(d+1)(d-3)}),\quad d>3,
\end{equation}
is a real quadratic subfield of $E_{\Pi(\bold{v})}$.

We will use this as an assumption to prove that no real fiducial vectors can exist in odd degree extensions of $\Q$.
 
\begin{proposition}
\label{prop:conditional_no_odd_degree_real_fiducials}
    Let $K \subset \R$ be a number field, and let $d > 3$. Assume that $\K_d \subseteq E_{\Pi(\bold v)}$ for every real fiducial vector $\bold v \in K^d$. Then no fiducial vector can exist if $[K:\Q]$ is odd.
\end{proposition}

\begin{proof}
    Suppose that $[K:\Q]$ is odd. By \Cref{subsec:odd_degree_K} we know that $d = 4m^2 - 1, m \in \Z_{\geq 1}$. This implies
    \begin{equation*}
        \K_d = \Q(\sqrt{(d+1)(d-3)}) = \Q(\sqrt{m^2-1}).
    \end{equation*}

    Set $F=\Q(\omega_d)$. Since $\Pi_{jl}=v_jv_l\in K$, we have $E_{\Pi(\bold v)}\subseteq KF$. Since $F/\Q$ is Galois and $K/\Q$ is finite, \cite[Proposition 14.4.19]{dummit2003abstract} guarantees that $KF/K$ is Galois, and that $\Gal(KF/K) \cong \Gal(F/K \cap F)$, so
    \begin{equation*}
        [KF:F] = \frac{[KF:K][K:\Q]}{[F:\Q]} = \frac{[F:K \cap F][K:\Q]}{[F:\Q]} = \frac{[K : \Q]}{[F \cap K : \Q]},
    \end{equation*}
    and since $[K:\Q]$ is odd, so is $[KF:F]$. The assumed containment $\K_d\subseteq E_{\Pi(\bold v)}\subseteq KF$ therefore forces $\K_d\subseteq F$, since otherwise $\K_dF/F$ would be a quadratic intermediate extension of $KF/F$.
    
    We will now show that this forces a contradiction. Since $d > 3$, we have $m > 1$ and from $(m-1)^2<m^2-1<m^2$, we know $m^2-1$ is not a square, so $m^2 - 1 = s^2D$ with $D>1$ squarefree. If $\Q(\sqrt D) \subseteq \Q(\omega_d)$, the conductor of this quadratic field, equal to the absolute value of its fundamental discriminant, must divide $d$ \cite[Chapter V, Remark 3.8 and Examples 3.10-3.11]{milneCFT}, and since $d$ is odd, we must have $D \equiv 1 \pmod4$. However, since $D \mid m^2 - 1$, we have
    \begin{equation*}
        D \mid \gcd(d,m^2-1) = \gcd(4m^2-1,m^2-1) = \gcd(3,m^2-1) \in\{ 1, 3\}.
    \end{equation*}
    Since $D > 1$, we have $D = 3$, which contradicts $D \equiv 1 \pmod4$. We conclude that $\K_d \not\subseteq F$, a contradiction.
\end{proof}

The proposition uses only the predicted containment
$\K_d\subseteq E_{\Pi(\bold v)}$, rather than the full description
of $E_{\Pi(\bold v)}$ as a class field. Establishing this containment
for arbitrary real fiducials would therefore remove the
odd-degree case unconditionally.

%
%
%

\section*{Acknowledgments}

The work of Igor Van Loo was supported by an NTU Provost Graduate Award, whose support is gratefully acknowledged.

%
%
%

\section*{ORCID}

Igor Van Loo - \href{https://orcid.org/0009-0001-0062-648X}{https://orcid.org/0009-0001-0062-648X}
\\
Fr\'ed\'erique Oggier -
\href{https://orcid.org/0000-0003-3141-3118}{https://orcid.org/0000-0003-3141-3118}

%
%
%

\section*{Appendix}

For the reader's convenience, we give a diagram which summarizes the main results of the paper.

\begin{figure}[h]
    \centering
    \begin{tikzpicture}

\node[draw]  (top) at (0,0) {$K \subset \R$};

\node[draw]  (r1c1) at (0,-1.5) {$d \equiv 3 \pmod{4}$};

\node[draw]  (r2c1) at (-5,-3) {$[K:\Q]$ odd};
\node[draw]  (r2c2) at (-2,-3) {$K \subset \Q^{tr}$};
\node[draw]  (r2c3) at (1,-3) {$[K:\Q]$ even};
\node[draw]  (r2c4) at (5,-3) {$\exists \Pf, e_{\Pf \mid 2}$ is odd};

\node[draw, align = center]  (r3c1) at (-5,-5) {$d = 4m^2 - 1$, \\ $m \in \Z_{\geq 1}$};
\node[draw, align = center]  (r3c2) at (-2.5,-5) {$d = 4m^2 - 1$, \\ $m \in \Z_{\geq 1}$};
\node[draw]  (r3c3) at (0,-5) {$4 \nmid [K:\Q]$};
\node[draw]  (r3c4) at (2,-5) {$\sqrt{d + 1} \notin \Q$};
\node[draw]  (r3c5) at (5,-5) {$v_2(d + 1) \equiv 2 \pmod{4}$};

\node[draw]  (r4c1) at (-5,-7) {$v_2(m)$ even};
\node[draw, align = center]  (r4c2) at (-0.1,-7) {$d = 4m^2 - 1$, \\ $m \in \Z_{\geq 1}$};
\node[draw, align = center]  (r4c3) at (2.1,-7) {$4 \mid [K:\Q]$, \\ $K \not\subset \Q^{tr}$};


\draw[->] (top) -- (r1c1) node[midway, right]{\Cref{theorem:real_fiducial_must_be_in_d=3mod4}};

\draw[dotted, ->] (r1c1) -- (r2c1);
\draw[dotted, ->] (r1c1) -- (r2c2);
\draw[dotted, ->] (r1c1) -- (r2c3);
\draw[dotted, ->] (r1c1) -- (r2c4);

\draw[->] (r2c1) -- (r3c1) node[midway, left]{\Cref{lemma:dimensions_4m2-1}};
\draw[->] (r2c2) -- (r3c2) node[midway, left]{\Cref{theorem:K_totally_real_implies_sqrtd1_in_Q}};
\draw[dotted, ->] (r2c3) -- (r3c3);
\draw[dotted, ->] (r2c3) -- (r3c4);
\draw[->] (r2c4) -- (r3c5) node[midway, right]{\Cref{prop:odd_ramification_index_implies_v2d+1equiv2mod4}};

\draw[->] (r3c1) -- (r4c1) node[midway, left]{\Cref{prop:degree_of_K_is_odd_implies_v2m_is_even}};

\draw[->] (r3c3) -- (r4c2) node[midway, left]{\Cref{prop:even_degree_not_div_by_4_implies_sqrtd1_in_Q}};
\draw[->] (r3c4) -- (r4c3) node[midway, right]{\Cref{theorem:K_totally_real_implies_sqrtd1_in_Q}};

\end{tikzpicture}
    \caption{Restrictions on the dimension $d$ in which a real fiducial vector can exist in, based on its coordinate field. Dotted arrows represent a subcase, full arrows represent an implication.}
    \label{fig:pd3mod4}
\end{figure}
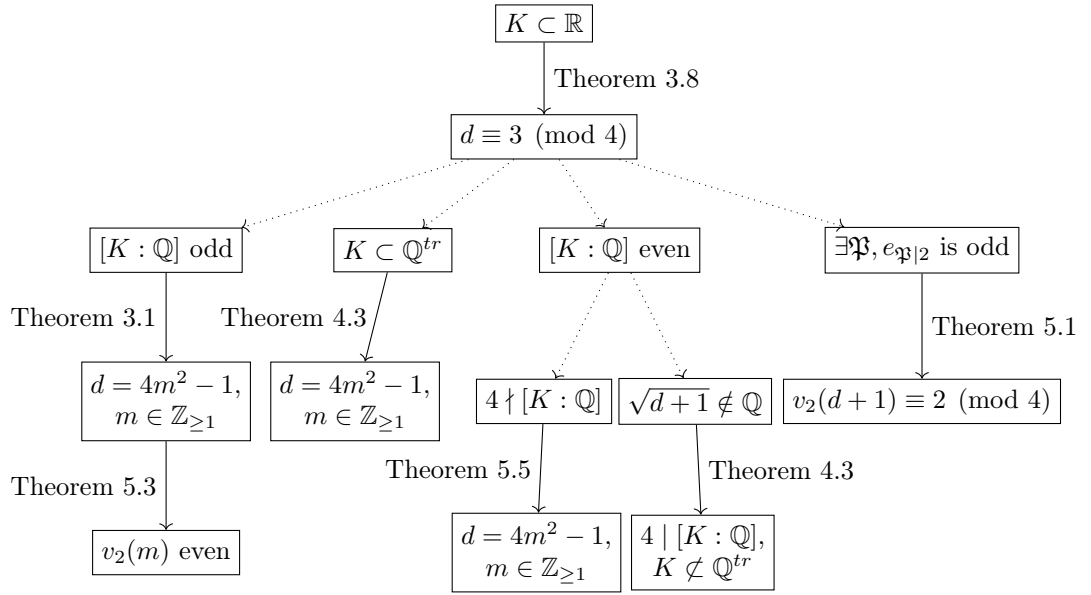

\bibliographystyle{plain}
\bibliography{references}

\end{document}